\documentclass[12pt]{article}
\usepackage{amssymb,amsmath, mathtools}

\usepackage[latin1]{inputenc}
\usepackage[T1]{fontenc}
\usepackage{play}
\usepackage{scalerel}
\usepackage{subcaption}

\usepackage{graphicx}
\DeclareGraphicsExtensions{.pdf,.png,.jpg}

\graphicspath{{images/}}

\usepackage{lipsum}

\let\OLDthebibliography\thebibliography
\renewcommand\thebibliography[1]{
  \OLDthebibliography{#1}
  \setlength{\parskip}{0pt}
  \setlength{\itemsep}{0pt plus 0.3ex}
}

\usepackage{amsmath}
\usepackage{amsthm}
\usepackage{amsfonts}
\usepackage{bbm}
\usepackage{amssymb}
\usepackage{graphicx}
\usepackage{color}

\usepackage{wasysym, marvosym}
\usepackage{lmodern}

\usepackage{fancybox}
\newlength{\querylen}
\newcommand{\prob}{\mathbb{P}}

\newcommand{\dt}{\,\mathrm{d}t}
\newcommand{\dx}{\,\mathrm{d}x}

\DeclareMathOperator\erf{erf}

\DeclareMathOperator\erfc{erfc}

\newtheorem{thm}{Theorem}

\newtheorem{assertion}[thm]{Proposition}
\theoremstyle{definition}

\theoremstyle{remark}

\begin{document}
\title{Running minimum in the best-choice problem} 
\author{Alexander Gnedin\thanks{Queen Mary, University of London}, \and Patryk Kozie{\l}\thanks{Wroc{\l}aw University of Science and Technology, Department of Fundamentals of Computer Science, Poland} \and and ~~ Ma{\l}gorzata Sulkowska$^{\dagger}$\thanks{Universit{\'e} C{\^o}te d'Azur, CNRS, Inria, I3S, France}}

\maketitle

\begin{abstract}
	We consider the best-choice problem for independent, not necessarily identically distributed observations $X_1, \cdots, X_n$ with the aim of selecting the sample minimum. We show that in this full generality the monotone case of optimal stopping holds and the stopping domain may be defined by the sequence of monotone thresholds. In the iid case we get the universal lower bounds for the success probability. We cast the general problem with independent observations as a variational first-passage problem for the running minimum process which simplifies obtaining the formula for success probability. We illustrate this approach by revisiting the full-information game (where $X_j$'s are iid uniform-$[0,1]$), in particular deriving new representations for the success probability and its limit by $n \rightarrow \infty$. Two  explicitly solvable models with discrete $X_j$'s are presented: in the first the distribution is uniform on $\{j,\cdots,n\}$,  and in the second the distribution is uniform on $\{1,\cdots, n\}$. These examples are chosen to contrast two situations  where the ties  vanish or persist in the large-$n$ Poisson limit.
\end{abstract}

\section{Introduction}
The best-choice aka secretary problems are stochastic optimisation tasks where the objective is to select one or few  `best' 
elements  of a random sequence observed in online regime.
Many versions of the problem are surveyed  in
 \cite{Ferguson, Handbook}, these differ in the way to compare data,   the sample size,   observer's information and other constraints on admissible decision strategies.
In this paper, under `best' we mean the sample minimum, 
admissible decision strategies are stopping times,
the number of observations is fixed and the data elements are drawn independently from distributions which may be different 
or have discontinuities. 
The focus will  be on  examples  highlighting a connection of the optimal stopping strategy with a first passage time for the  running minimum process.

Let $X_1,\cdots,X_n$ be  independent random variables
and denote 
$M_j=\min \{X_1,\cdots,X_j\}$, $j\leq n,$
the running minimum. 
Suppose the values  of $X_j$'s are observed sequentially with the objective to stop, with no recall permitted, at some value coinciding with the ultimate minimum $M_n$, considered
 as the best option (possibly not unique) out of $n$ options available.
Specifically we are concerned with the maximum probability
\begin{eqnarray}\label{w}
v_n:=\sup_\tau \prob(X_\tau=M_n),
\end{eqnarray}
where $\tau$ runs  over all stopping times  adapted to the  natural filtration of  the observed sequence. 
The general discrete-time optimal stopping theory \cite{CRS} ensures that there exists a stopping time achieving $v_n$.

We call
$X_j$  a (lower) record in the event $X_j=M_j$. To include the possibility of ties 
we make difference 
between strict records $X_j<M_{j-1}$ and weak records $X_j=M_{j-1}$ 
(where formally $M_0=\infty$). 
Given a  record  $X_j=x$,   
the probability of  $X_j=M_n$  decreases in  
$x$  and increases in $j$, hence it is natural to expect  
that the optimal  stopping time can be  determined in terms of certain critical thresholds 
$b_1\leq\cdots \leq b_n=\infty$. 
For the case of iid observations drawn  from given continuous distribution 
this was shown in the seminal paper by Gilbert and Mosteller \cite{Gilbert_Mosteller},
where the problem was called
 the {\it full-information} game.
The name underscores
 the difference with
the  {\it no-information} problem where the distribution is  unknown and the stopping decisions depend only on the  relative ranks of  observations.
In the iid continuous case 
the optimal value, denoted further $\underline{v}_n$, does not depend on the distribution which may be assumed uniform-$[0,1]$. 
For sampling from the uniform distribution Gilbert and Mosteller \cite{Gilbert_Mosteller}
derived exact and approximate formulas  for the thresholds and used these to observe numerically that $\underline{v}_n$ converge 
to  $\underline{v}=0.580164\cdots$. The limit was confirmed by  Samuels \cite{Handbook} who actually proved that the convergence is monotone, $\underline{v}_n\downarrow\underline{v}$, and derived
an  explicit formula for $\underline{v}$.
Further insight  
was gained from coupling with the  asymptotic `$n=\infty$' form of the problem associated with a homogeneous planar Poisson process,
see \cite{GnedinFI, GnedinPlanar}.

Some past work is related   to the best-choice problem with discrete or  not identically distributed observations.
Campbell \cite{Campbell} considered iid discrete observations from a Dirichlet process,  a setting which according to \cite{Ferguson, Handbook} should be classified as a  {\it partial information} 
version of the problem, where the decision process incorporates Bayesian inference about the source distribution.
Faller and R{\"u}schendorf \cite{FR} extended the  framework  of \cite{GnedinFI} to connect the finite-$n$ problem with independent observations 
to a possibly nonhomogeneous planar Poisson limit.
Hill and Kennedy \cite{HK} used single-threshold strategies to  prove a general sharp bound $v_n\geq (1-1/n)^{n-1}$ (implying the lower bound $1/e$ uniformly in $n$)
and found explicitly the worst-case  distributions  for independent non-iid  $X_j$'s.
The same bound is known to hold if  the data  are sampled from a continuous distribution but the observation order is controlled by adversary \cite{GnKr}.  
Of recent results we mention the paper by
Nuti \cite{Nuti}, where  
the lower bound $\underline{v}_n$  was shown for the model where the observed sequence is a random permutation of $n$ independent 
non-iid values.

The rest of the paper is  organised as follows.
In Section 2 we  treat the best-choice problem with independent observations  in full generality, argue that the monotone case of optimal stopping \cite{CRS} holds, hence  conclude that 
the structure of the stopping domain is  as in  the iid continuous case of \cite{Gilbert_Mosteller}.
For the iid case with ties we prove the universal lower bounds
 $\underline{v}_n$ and  $\underline{v}$.
Then we cast the general problem with independent observations as
 a variational first-passage problem for the running minimum process.
This  gives some simplification in assessing the success probability, since the conditional payoffs can be expressed directly in terms of the Markov process  $(M_j,~j\leq n)$.
In Section 3 we illustrate  the approach by revisiting the  classic iid continuous case, in particular
deriving new   representations for  $\underline{v}_n$ and $\underline{v}$.
 In  Sections 4 and 5  we scrutinise 
two  explicitly solvable 
models with discrete $X_j$'s: in the first,  suggested to us by Prof. J. Cicho{\'n}, the distribution  
is uniform on $\{j,\cdots,n\}$ (the triangular model),  and in the second
the distribution is uniform on $\{1,\cdots, n\}$ (the rectangular model). 
These examples are chosen to contrast two situations  where the ties  vanish or persist in the large-$n$ Poisson limit,
and both can be included in natural parametric families of distributions which are still tractable.

\section{Some general  facts}


\subsection{Structure of the stopping domain}
Assuming $X_1,\cdots,X_n$ independent,
let $F_j$  denote the right-continuous  cdf of $X_j$, and let
\begin{eqnarray*}
s(j,x)&:=&\prod_{k=j+1}^n (1-F_k(x-)),\\
v(j,x)&:=&\sup_{\{\tau:~\tau> j,~X_\tau\leq x\}} \prob( X_\tau=\min\{X_{j+1},\cdots,X_n\}).
\end{eqnarray*}
If no choice has been made from the first $j-1$ observations, then
conditionally on record $X_j=x$   stopping is successful 
with probability $s(j,x)$, and the continuation value achievable by skipping the record is $v(j,x)$, 
regardless of  the past observed values.
 Let $B:=\{(j,x): s(j,x)\geq  v(j,x)\}$, seen as a subset of $\{1,\cdots,n\}\times {\mathbb R}$.
 The optimality principle for problems with finitely many decision steps \cite{CRS}
entails that the  stopping time 
$$
\tau_n:=\min\{j\leq n:~ X_j=M_j,   ~(j, X_j)\in B\}
$$
($\min\varnothing=n$) is optimal.

The function $s(j,x)$ is nonincreasing and left-continuous in $x$, and nondecreasing in $j$.
Likewise, $v(j,x)$  is  nondecreasing and right-continuous in $x$.
Discontinuity 
may only occur if some of  the distributions $F_{j+1},\cdots,F_n$ have an atom at $x$, in which case the jump  of $s(j,x)$
is equal to the probability of $\min\{X_{j+1},\cdots,X_n\}=x$,  and the jump   of $v(j,x)$ is not  bigger than that.
By the monotonicity,  there exists a threshold $b_j$ (finite or infinite)
such  that $s(j,x)\geq v(j,x)$ for $x< b_j$ and $s(j,x)<v(j,x)$ for $x>b_j$, so if the threshold is finite
  $\{x: s(j,x)\geq v(j,x)\}$ is either $(-\infty,b_j)$ or $(-\infty, b_j]$.

\paragraph{Example}
The  {\it Bernoulli pyramid}  of Hill and Kennedy  \cite{HK} illustrates an extreme possibility.
This has $X_1=1$, and for $j = 2, \cdots, n$ the observations are 
 two-valued  with
$\prob(X_j=1/j)=1-\prob(X_j=j)=p$. 
The relative ranks assume only extreme values and are independent.
From this one computes readily that
 $s(j,x)=(1-p)^{n-j}$ and  $v(j,x)=(n-j)p(1-p)^{n-j-1}$,
hence
 it is optimal to stop if $1-p\geq (n-j)p$. That is to say,
 $b_j=-\infty$  for $j<   n-   (1-p)/p$ and $b_j=\infty$  for $j\geq  n-   (1-p)/p$. 
The worst-case scenario for the observer is 
the parameter value $p=1/n$, when the optimal best-choice probability  is $v_n=(1-1/n)^{n-1}$.
\vskip0.2cm

\begin{assertion}\label{as2}
The optimal thresholds satisfy $b_1\leq\cdots\leq b_n=\infty$,  and 
$B$ is a closed (no-exit) set for the running minimum process.  
\end{assertion}
\begin{proof}
Arguing by backward induction 
suppose  $b_{j+1}\leq\cdots\leq b_n=\infty$
and that $B$ is no-exit for the paths  entering on step $j+1$ or later. 
If  $1-F_{j+1}(b_{j+1})=0$  then obviously $s(j,x)=0$ for $x>b_{j+1}$,  whence $b_{j}\leq b_{j+1}$.
Suppose $1-F_{j+1}(b_{j+1})>0$. 
For $x>b_{j+1}$ a continuity point of $F_{j+1}$ we have 
\[
\begin{split}
	v(j,x) & \geq 
	[1-F_{j+1}(x)]v(j+1,x)\\
	& + \int_{-\infty}^xs(j+1,y){\rm d} F_{j+1}(y)\\
	& \geq s(j+1,x)\geq s(j,x),
\end{split}
\]
and $v(j+1,x)> s(j+1,x)$ implies that the inequality is strict whenever  $F_{j+1}(x)<1$, hence $b_j\leq x$.  Letting  $x\downarrow b_{j+1}$ gives  $b_j\leq b_{j+1}$,
which by the virtue  of natural monotonicity of the running minimum yields the induction step if   either  $b_j<b_{j+1}$, or  $b_j=b_{j+1}$  and $s(j,b_j)< v(j,b_j)$.

It remains to exclude the possibility  $b_j=b_{j+1}$ when stopping 
at record $X_k=b_j$ is optimal for $k=j$ and not optimal for $k=j+1$.
But
in the case $b_j=b_{j+1}$ and $s(j,b_j)\geq v(j,b_j)$, for $z=b_j$ we obtain
\[
\begin{split}
	s(j+1,z)\geq s(j,z)  & \geq  v(j,z)\\
	& = [1-F_{j+1}(z)] v(j+1,z)\\
	& + \left[F_{j+1}(z)-F_{j+1}(z-)\right] \max \{s(j+1,z),v(j+1,z)\} \\
	& + \int_{-\infty}^{z-} s(j+1,y){\rm d} F_{j+1}(y)\geq v(j+1,z),
\end{split}
\]
therefore if stopping at record $X_j=b_j$ is optimal this also holds for $X_{j+1}=b_j$.
This completes the induction step.
\end{proof}

Implicit to the proof is the formula for the continuation value in $B$
\begin{equation}\label{v-in-B}
v(j,x)=\sum_{k=j+1}^n \left( \left( \prod_{i=j+1}^{k-1}(1-F_i(x)) \right) \int_{-\infty}^{x} s(k,y) \,{\rm d}F_k(y) \right), ~~(j,x)\in B.
\end{equation}

 Extending our notation we may understand the running minimum $(M_j,  ~j\leq n)$ as a transient Markov 
chain with two types of states. A state $(j,x)$ is associated with the event $X_j>M_j=x$ and state $(j,x)^\circ$ with the record $X_j=M_j=x$.
The transition probabilities are straightforward in terms of the $F_j$'s, for instance a transition from either $(j,x)$ or $(j,x)^\circ$ to $(j+1,x)^\circ$ 
occurs with probability $F_{j+1}(x)-F_{j+1}(x-)$.

Accordingly, we say that a first passage into $B$ occurs {\it by jump} if the running minimum enters the set at some record time (in a state $(j,x)^\circ$), and we speak of the first passage {\it by drift}
 otherwise (i.e. in a state $(j,x)$).
The success probability is $s(j,x)$ if the first passage occurs in $(j,x)^\circ$ and it is $v(j,x)$ given by (\ref{v-in-B}) for  $(j,x)$.
The   discrimination of first passage cases
leads to a two-term decomposition of $v_n$, which we will derive in the sequel for concrete examples.

\subsection{Bounds in the iid case}

Suppose $X_1,\cdots,X_n$ are iid with arbitrary distribution $F$. 
By breaking ties with the aid of auxiliary randomisation   we shall connect the problem to the iid continuous case.

It is a standard fact, known as probability integral transform, that
$F(X_j)$ has uniform distribution if $F$ is continuous.
In general $F(X_j)$ has a  `uniformised'  distribution on $[0,1]$  which does not charge the open gaps in the range of $F$ 
\cite{RCS}. To spread a mass over the gaps, let $U_1,\cdots,U_n$ be iid uniform-$[0,1]$ and independent of $X_j$'s then we will have that
the random variables
\begin{equation}\label{YX}
Y_j:=F(X_j)- [F(X_j)-F(X_j-)]U_j
\end{equation}
are also iid uniform-$[0,1]$.
Moreover, they satisfy $F^{\leftarrow}(Y_j)=X_j$ where $F^{\leftarrow}$ is the generalised inverse.
By monotonicity, $Y_j=\min\{Y_1,\cdots,Y_n\}$ implies $X_j=M_n=\min\{X_1,\cdots,X_n\}$ therefore 
for any stopping time $\tau$ adapted to the $Y_j$'s
$$
\prob(Y_\tau=\min\{Y_1,\cdots,Y_n\})\leq \prob(X_\tau=M_n)\leq v_n.
$$
The second inequality follows since such $\tau$ is a randomised stopping time hence cannot improve upon the nonrandomised optimum \cite{CRS}.
Taking supremum in the left-hand side gives $\underline{v}_n\leq v_n$.

The probability of a tie for the minimum in  sequence $X_1,\cdots,X_n$ is given by

\begin{eqnarray}\label{delta}
\delta_n:=1- n \int_{-\infty}^\infty  (1-F(x))^{n-1}{\rm d}F(x).
\end{eqnarray}
If there is no tie for the minimum, that is $X_j=M_n$ holds for exactly one $j\leq n$, then
$X_j=M_n$ implies $Y_j=\min\{Y_1,\cdots,Y_n\}$. 
Thus we obtain
$$\prob(X_\tau=M_n)-\delta_n\leq \prob(X_\tau=M_n,~{\rm no ~tie})\leq \prob(Y_\tau=\min\{Y_1,\cdots,Y_n\})\leq \underline{v}_n$$
for each $\tau$ adapted to the $X_j$'s (hence also adapted to the $Y_j$'s). 
Choosing the optimal $\tau=\tau_n$  in the left-hand side gives  $v_n\leq \underline{v}_n+ \delta_n$.
To summarise, with the account of monotonicity of  $\underline{v}_n$'s \cite{Handbook} we have the following result.
\begin{assertion} \label{as_bounds_iid} In the iid case
$
\underline{v}_n\leq v_n\leq \underline{v}_n+\delta_n,
$
 hence $\underline{v}_n>\underline{v}$
implies the universal sharp bound $v_n>\underline{v}$.
\end{assertion}

\noindent
The probability of a tie for the first place (maximum)  is a thoroughly studied subject \cite{Tie}. With the obvious adjustment
these results are applicable  
 to infer about possible asymptotic behaviours of (\ref{delta}).

\subsection{Poisson limit}

Let $N$ be a Poisson point process in $[0,T)\times {\mathbb R}$ with $0<T\leq \infty$ and some  nonatomic intensity measure $\mu$, which satisfies 
\begin{eqnarray*}
\mu(\{t\}\times{\mathbb R})&=&0,~~0\leq t<T,\\
\mu([t,T)\times [0,\infty))&=&\infty,~ 0<t<T,\\ 
\mu([0,t)\times [0,\infty))&=&\infty,~ 0<t<T,\\ 
\mu([0,T)\times (-\infty,x])&<&\infty, ~x\in {\mathbb R}.
\end{eqnarray*}
The generic point $(t,x)$ of $N$ is thought of as 
mark $x$ observed at time $t$; this is  considered  a strict record if $N([0,t)\times [0,x])=1$ (point $(t,x)$ itself) and weak record if $N([0,t)\times [0,x))=0$.
The first assumption excludes multiple points, but several points of $N$ can share the same mark $x$ if 
the measure is singular with
$\mu([0,T)\times\{x\})>0$.
Furthermore, there is a well defined 
 running minimum process $(Z_t,~t\in[0,T))$, where $Z_t$ is the minimal mark observed within time interval $[0,t]$.
The  information of observer accumulated by time $t$ is  the configuration of  $N$-atoms on $[0,t]\times {\mathbb R}$.
The task is to stop with the highest possible probability at observation with the minimal mark $Z_{T}:=\lim_{t\to T} Z_t$.

We refer to  Faller and R{\"u}schendorf \cite{FR} for more formal and complete treatment of the optimal stopping problem under slightly different assumptions on $\mu$.
In particular, they only consider the case $T=1$, but this is not  a substantial constraint, because increasing transformations of  scales 
do not really change the problem.
Part (a) of their  Theorem 2.1 implies 
that there exists a nondecreasing  function $b: [0,T)\to (-\infty,\infty]$ such that it is optimal to stop at the first record 
falling in the domain $B=\{(t,x): ~x\leq b(t)\}$. Equation (2.7) of \cite{FR} gives a 
multiple integral expression  for the probability of success under the optimal stopping time.

A connection 
with the discrete time stopping problem is the following. Let $X_1,\cdots,X_n$ be independent,  possibly with different  distributions that may depend on $n$.  
Consider the scatter of $n$ points 
$\{(1, X_1),\cdots, (n, X_n)\}$,  subject to   a suitable monotone coordinate-wise scaling, as a finite point process $N_n$ on the plane, and suppose that $N_n$   converges 
weakly to $N$.
Then,
by part (b)  of the cited theorem from \cite{FR}, $v_n$ converges to the optimal stopping value for $N$.
Part  (c) asserts that stopping on the first record of $N_n$
falling in $B$ is asymptotically optimal for the discrete time  problem.

The choice of scaling is dictated by the vague convergence of the  intensity measure ${\mathbb E} N_n(\cdot)$ to a Radon measure on the plane.
In the iid case this is  typically a linear scaling from the extreme-value theory \cite{Resnick}.
See \cite{Falk} for examples of  scaling for non-iid  data.

\section{The full-information game revisited}

\subsection{The discrete time problem}

Let $X_1, \cdots, X_n$  be iid uniform-$[0,1]$. Thus ties have probability zero and all records are strict. 
Consider stopping time of the form
$$\tau=\min\{j\leq n: X_j\leq b_j,\, X_j=M_j\},$$
where $0\leq b_1\leq\cdots\leq b_n\leq 1$ are arbitrary thresholds, not necessarily optimal.
We aim to decompose the success probability  ${\mathbb P}(X_\tau=M_n)$ according to the distribution of the  first passage time for the running minimum
$$\sigma:=\min\{j\leq n:~ M_j\leq b_{j+1}\}.$$
Obviously $\sigma\leq\tau$: namely $\sigma=\tau$ if $\sigma$ is a record time, and otherwise $\sigma<\tau$ and $\tau$ is the first record time (if any) after $\sigma$.

Suppose $\sigma=j, M_j=x$. If $x\leq b_j$ then $M_{j-1}\leq b_j$ is impossible (otherwise we would have $\sigma<j$), hence $M_{j-1}>b_j$ implying that $\tau=j$ is a record time;
we qualify this event as a passage by jump.
Alternatively, in the case $b_j<x\leq b_{j+1}$ the passage is by drift and $\tau>j$ coincides with the first subsequent record time (if any).
Accordingly, we have for the joint distribution of 
$\sigma$ and  $M_\sigma$ 
$$
{\mathbb P}(\sigma=j\,,\, M_{j}\in [x, x+{\rm d}x])=\begin{dcases*} 
~~(1-b_j)^{j-1}\, {\rm d} x,  ~~~  0\leq x \leq b_j,\\
~j (1-x)^{j-1}\,{\rm d} x , ~~~   b_j<x\leq  b_{j+1}.
\end{dcases*} $$  
In the first case the (conditional) success probability is $(1-x)^{n-j}$, 
and in the second case 
$$\sum_{k=j+1}^{n} (1-x)^{k-j-1} \int_0^x  (1-y)^{n-k} {\rm d} y  =  
\sum_{k=1}^{n-j} {[(1-x)^{n-j-k}-  (1-x)^{n-j}]}/{k}.$$
Integrating yields 
\begin{eqnarray*}
{\mathbb P}(\sigma=j,~   X_\tau=M_n)&=&\\
\frac{(1-b_j)^{j-1}-(1-b_j)^{n}}{n-j+1}&+&
j \sum_{k=j}^{n-1} \left[\frac{(1-b_j)^{k}-(1-b_{j+1})^{k}}{k(n-k)}\,-\,\frac{(1-b_j)^n-(1-b_{j+1})^n}{n(n-k)}   \right],
\end{eqnarray*}
which can be compared with Equation (3c-1) from \cite{Gilbert_Mosteller} (where $d_j$ is our $1-b_j$) for $\prob(\tau=j, X_j=M_n)$.
Summation gives the success probability
\begin{equation}\label{GM}
{\mathbb P}( X_\tau=M_n)= \frac{1}{n}\left[1 -\sum_{j=1}^n {(1-b_j)^n}\right] + \sum\limits_{j=1}^{n-1}\sum_{i=1}^j \left[\frac{(1-b_i)^j}{j(n-j)}-\frac{(1-b_i)^n}{n(n-j)}\right],
\end{equation}
with two parts corresponding to the passage by jump and drift.

The {\it optimal} threshold $b_j$  is a solution  to 
$$\sum_{i=1}^{n-j} \frac{(1-x)^{-i}-1}{i}=1,~~~~x\in[0,1].$$
Using this Sakaguchi \cite{Sak} obtained a nice formula
$$\underline{v}_n= \frac{1}{n}\left(1+\sum_{j=1}^{n-1}\sum_{k=j}^{n-1} \frac{(1-b_j)^k}{j}\right),$$
which also gives the mean ${\mathbb E}[\tau_n/n]$ for the optimal $\tau_n$, as was shown by Tamaki \cite{Tamaki}.

\subsection{Poissonisation} \label{sec:GMfull}

Let $N$ be a planar Poisson process in $[0,1]\times[0,\infty)$
with the Lebesgue measure as intensity.
It is well known and easy to verify that
for $X_1,\cdots,X_n$ iid uniform-$[0,1]$, 
the planar point process with  atoms $(j/n, nX_j), j\leq n$, converges weakly to $N$. 

There exists a still closer  connection through 
embedding of the finite sample in $N$ \cite{GnedinFI}.  Consider  
instead of  the uniform distribution 
the mean-$n$ exponential. A sample from this can be implemented by splitting $[0,1]$ evenly in $n$ subintervals, and
 identifying the $j$th point of  $N_n$  with the point of $N$ having the minimal 
 mark among the arrivals within the time interval $[(j-1), j/n)$.
By this coupling the inequality $\underline{v}_n>\underline{v}$ becomes immediate by noting that the stopping problem for $N_n$ is equivalent to a
choice  from $N$ where the observer at time $j/n$ has the power to revoke  any arrival within $[(j-1)/n, j/n)$.

The running minimum $(Z_t, ~t\in[0,1))$ is defined to be the lowest mark of arrival on $[0,t]$.
With every point $(t,x)\in [0,1]\times [0,\infty)$ we associate a
rectangular {\it box} $([t,1]\times[0,x])$ to the south-east with the corner point $(t,x)$ excluded.
 If  there is an atom of $N$ at location $(t,x)$ we regard it as record if there are no other atoms south-west of $(t,x)$, and it is the last record (with mark $Z_{1})$
if the box  contains no atoms. 
(See Figure~\ref{fig:record_rect}.)

\begin{figure}[!ht]
	\centering{
		\includegraphics[width = 0.4 \textwidth]{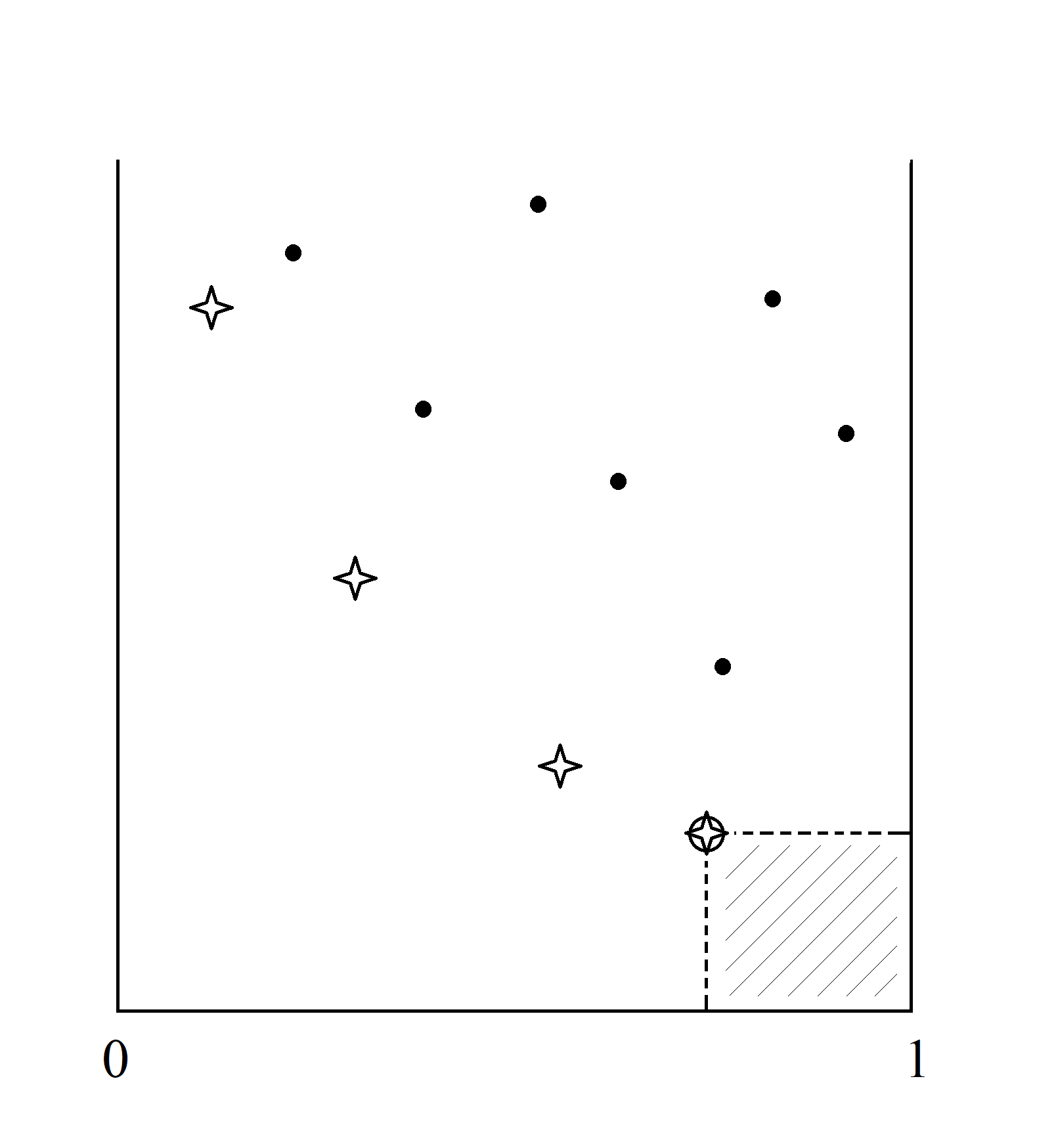}
		\caption[]{Records in a planar Poisson process in $[0,1]\times [0,\infty)$; 	\scalerel*{\includegraphics{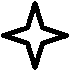}}{B} - record (no Poisson atoms south-west of it), 
		\scalerel*{\includegraphics{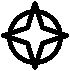}}{B} - the last record (its box contains no Poisson atoms).} \label{fig:record_rect}
	}
\end{figure}


An {\it admissible} strategy is a stopping time which takes values in the random set of record times 
and may also assume value $1$ (the event of no choice). With every nondecreasing continuous 
function $b:[0,1]\to[0,\infty)$, $b(1-)=\infty$,
we associate a stopping time
$$\tau=\inf\{t: (t, Z_t){\rm ~~is ~ record~},~Z_t\leq b(t)\},$$
where $\inf\varnothing=1$.
The associated  first passage time into $B$
 is defined as 
$$\sigma=\inf\{t: ~Z_t\leq b(t)\}.$$
Clearly, $\sigma$ is adapted to the natural filtration of $N$, $\sigma<1$ a.s. and $\sigma\leq \tau$. However, $\sigma$ is not admissible,
because in the event $Z_\sigma=b(\sigma)$ of the boundary crossing by drift,    there is  arrival at time $\sigma$ 
with probability  zero.

The modes of first passage by jump or drift 
can be distinguished in geometric terms. Move the  rectangular frame spanned on $(0,0)$ and $(t, b(t))$ 
until one of the sides meets a point of $N$. If the point falls on the eastern side of the frame, the running minimum crosses the boundary  by a jump and $\sigma=\tau$,
if the point appears on the northern side, the boundary is hit by a drift  and $\sigma<\tau$. See Figures \ref{fig:sigma_is_tau} and \ref{fig:sigma_less_tau}.

\begin{figure}[ht]
	\begin{center}
		\begin{minipage}[b]{0.45\textwidth}
			\includegraphics[width=1.0\textwidth]{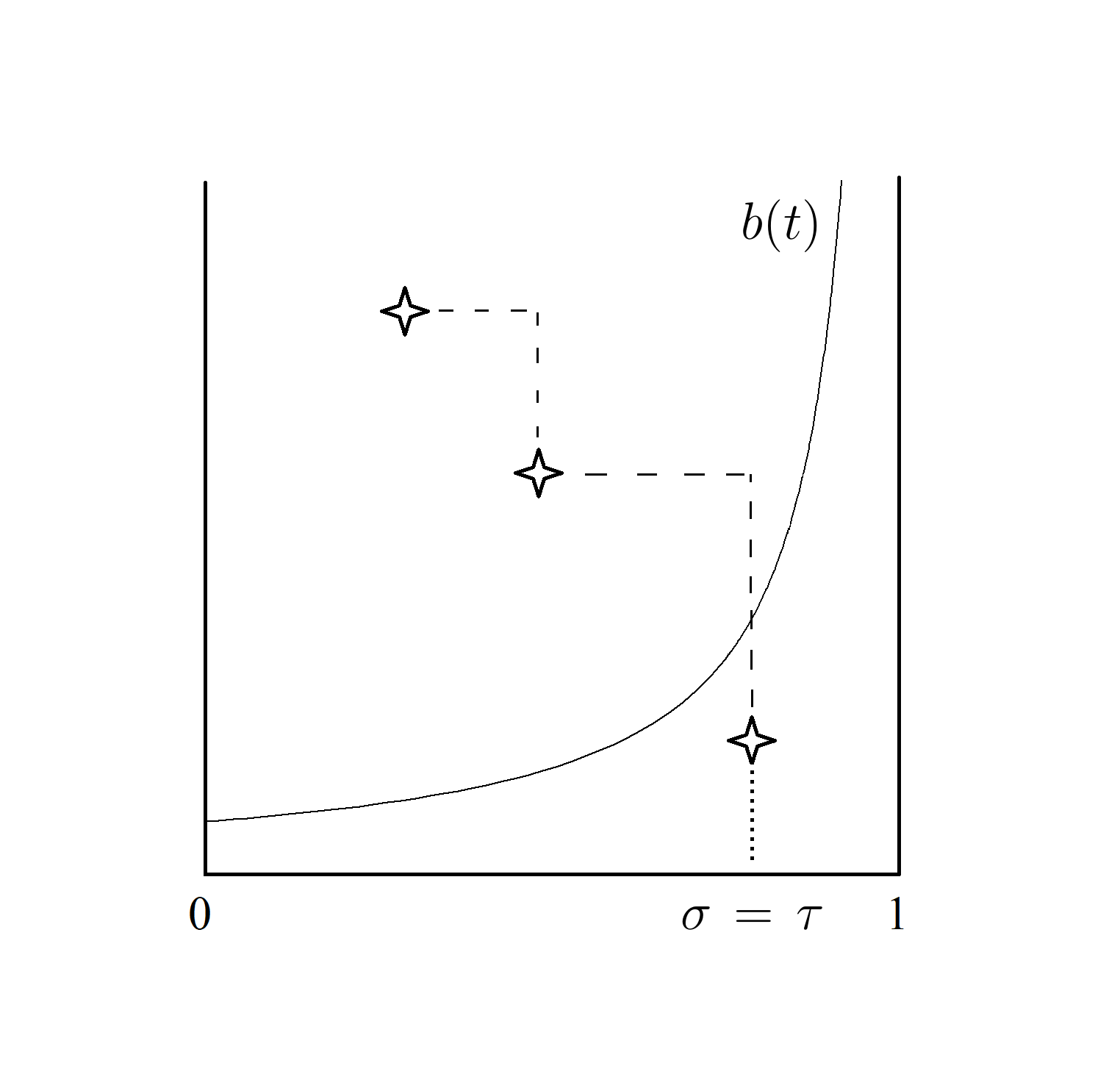}
		\end{minipage}
		\begin{minipage}[b]{0.45\textwidth}
			\includegraphics[width=1.0\textwidth]{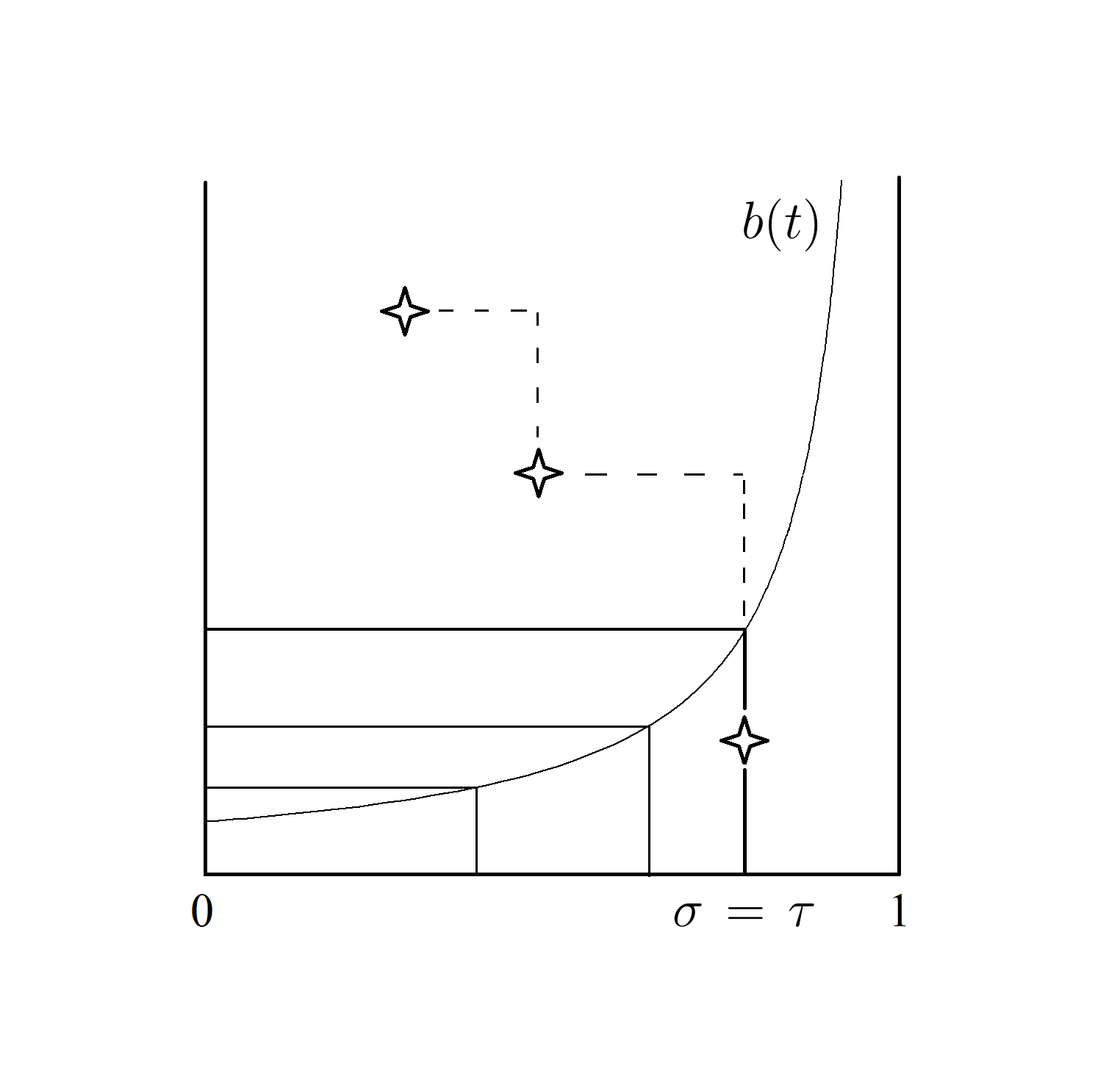}
		\end{minipage}
		\caption{Event $\sigma=\tau$,  the boundary $b(t)$ is crossed by a jump, hence the record is  caught by the eastern side of a frame spanned on $(0,0)$ and $(t,b(t))$.} \label{fig:sigma_is_tau} 
	\end{center}  
\end{figure}

\begin{figure}[ht]
	\begin{center}
		\begin{minipage}[b]{0.45\textwidth}
			\includegraphics[width=1.0\textwidth]{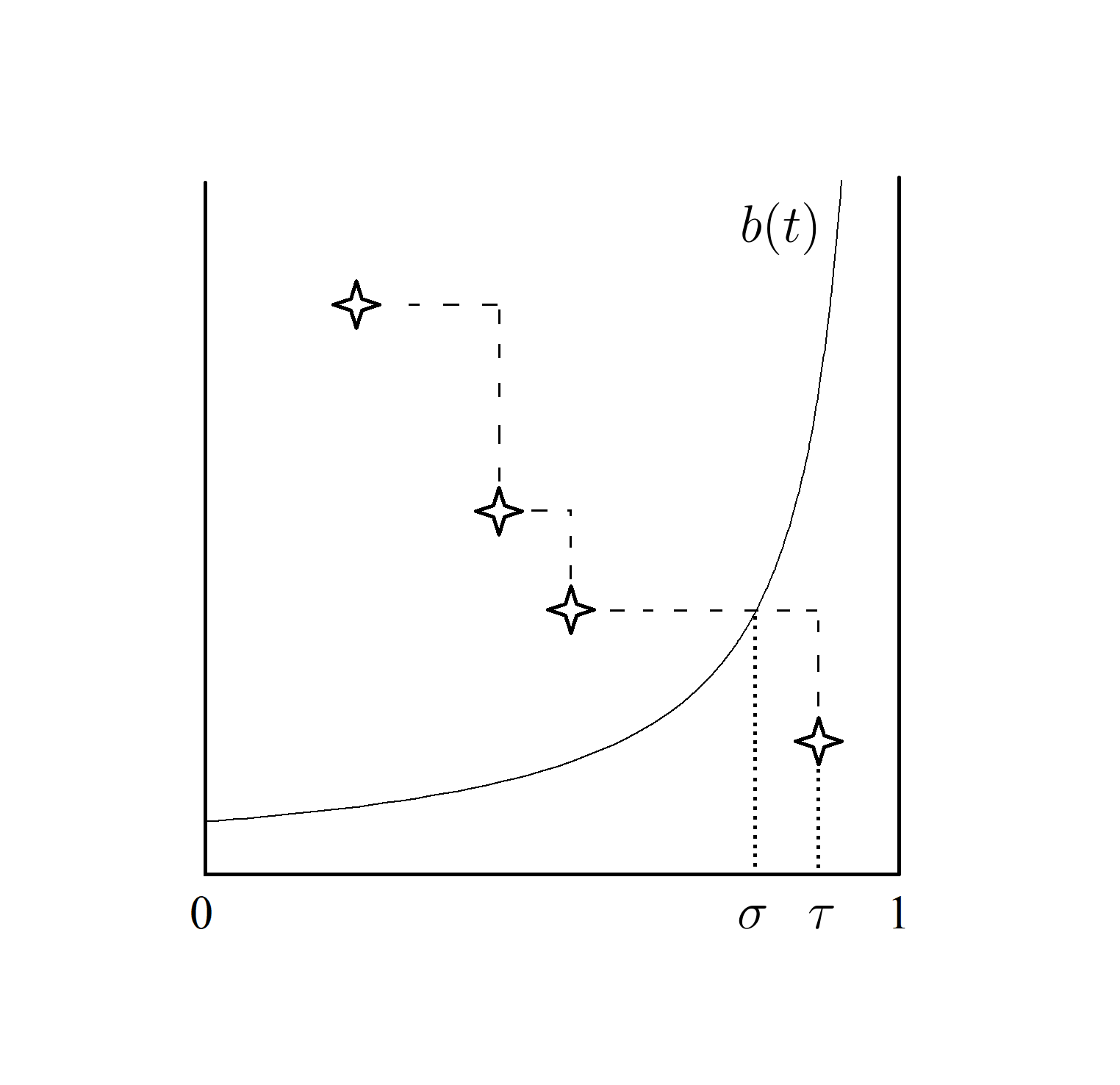}
		\end{minipage}
		\begin{minipage}[b]{0.45\textwidth}
			\includegraphics[width=1.0\textwidth]{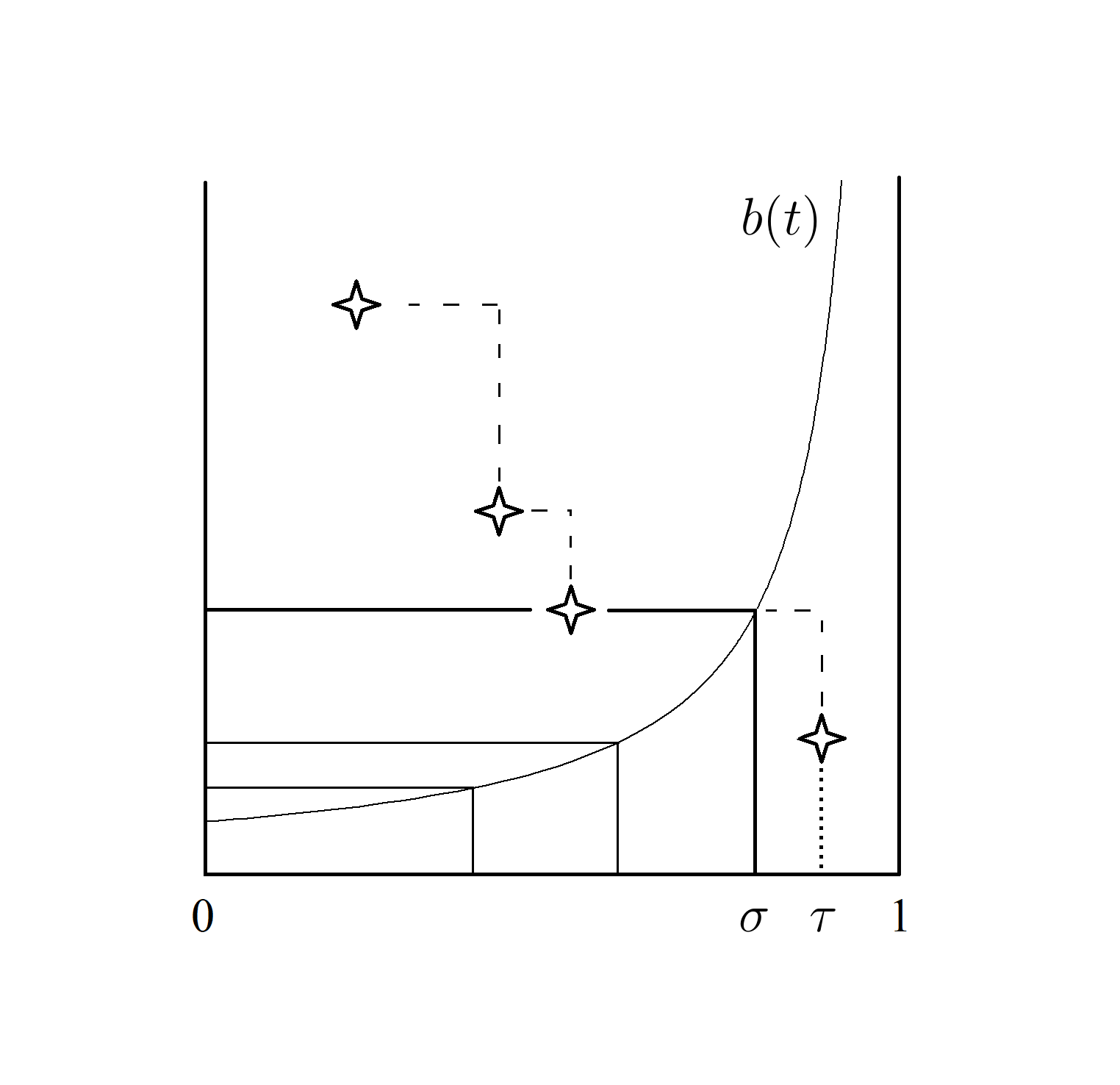}
		\end{minipage}
		\caption{Event $\sigma<\tau$ - the boundary $b(t)$ hit by a drift. Record caught by the northern side of a frame spanned on $(0,0)$ and $(t,b(t))$.} \label{fig:sigma_less_tau} 
	\end{center}  
\end{figure}

We aim to express the success probability of $\tau$ in terms of $(\sigma, Z_\sigma)$. A key observation which leads to explicit formulas is the 
self-similarity property: two boxes with the same area $z$ can be mapped 
to one another by a bijection which preserves both measure and coordinate-wise order.
Consider a box with apex at $(t,x)$ hence size $z=(1-t)x$. Then

\begin{itemize}

\item[(i)] if stopping occurs on a record at $(t,x)$  it is successful with probability $e^{-z}$,

\item[(ii)] if stopping occurs on a record  at $(t, Ux)$, for $U$ distributed uniformly on $[0,1]$,
 it is successful with probability 
$$J(z):=\int_0^1 e^{-zu}\,{\rm d}u= \frac{1-e^{-z}}{z},$$

\item[(iii)] if stopping occurs at the earliest arrival  inside the box (if any)
 it is successful with probability 
$$D(z):=\int_0^z   e^{-s}J(z-s)\,{\rm d}s= 
e^{-z}\int_0^z \frac{e^{s}-1}{s}{\rm d}s.$$
\end{itemize}
These formulas are most easily derived for {\it standard} box $[0,z]\times[0,1]$.

Now,  the running minimum crosses the boundary $b$ at time $[t,t+{\rm d}t]$ by jump 
(hence $\sigma=\tau$) if
 there are no Poisson
atoms south-west of the point $(t,b(t))$, and there is an arrival below $b(t)$.
Given such arrival, the distribution of the record value $Z_t$ is
uniform on $[0, b(t)]$, therefore  this crossing event contributes to the success probability
$$ e^{-t b(t)} b(t)\, J((1-t)b(t))\,{\rm d}t.$$
Alternatively, $(Z_t)$ drifts into the stopping domain at time $[t,t+{\rm d}t]$ (hence $\sigma<\tau)$ and $\tau$ wins with the next available record with probability
$$ e^{-t b(t)} t (b(t+{\rm d}t)-b(t)) D((1-t)b(t)).$$
We write   the success probability with $\tau$ as a functional of the boundary 

\begin{equation}\label{CV}
{\cal P}(b):= \int_0^1 e^{-t b(t)} b(t)\, J[(1-t)b(t)]\,{\rm d}t+\int_0^\infty 
e^{-t b(t)} t \, D[(1-t)b(t)] \,{\rm d}b(t).
\end{equation}
Note that the distribution of $\sigma$ is given by
$${\mathbb P}(\sigma\leq t)=\int_0^t e^{-s b(s)} b(s)\,{\rm d}s+\int_0^{b(t)} 
e^{-s b(s)} s \,  \,{\rm d}b(s),$$
where the terms correspond to two types of boundary crossing.

We may view maximising the functional
 (\ref{CV})
as a problem from the calculus of variations.
Recalling that the box area at record arrival is the only statistic which matters, suggests  to try the  hyperbolic shapes
\begin{equation}\label{hyperb}
b(t)=\frac{{\beta}}{(1-t)}.
\end{equation}
Indeed,  equating (i) and (iii), $e^{-z}=D(z)$, we see that the balance between immediate stopping and stopping at the next record is achieved
for $\beta^*=0.804352\cdots$ solving the equation 
$$ \int_0^z \frac{e^s-1}{s}{\rm d s}=1.$$

The optimal stopping time is defined by domain $B$ with the hyperbolic boundary and $\beta^*$ (see Figure~\ref{fig:boundary_rect}), because
$B$ is no-exit domain for the running minimum process, hence the monotone case of optimal stopping  \cite{CRS} applies.

\begin{figure}[!ht]
	\centering{
		\includegraphics[width = 0.5 \textwidth]{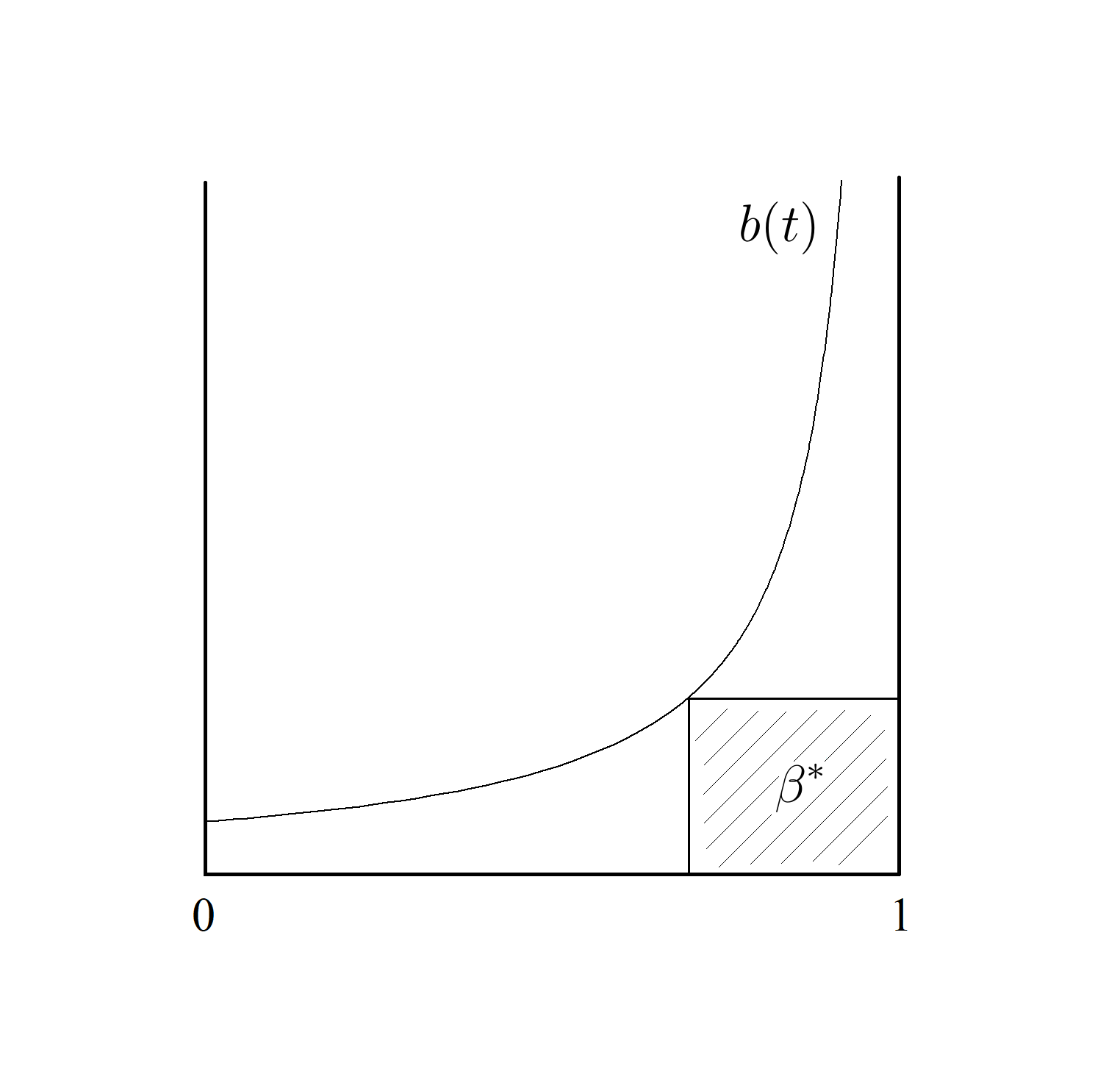}
		\caption{The hyperbolic boundary $b(t) = \frac{\beta^*}{1-t}$.} \label{fig:boundary_rect}
	}
\end{figure}

For  (\ref{hyperb}), the functional simplifies enormously, and the success probability becomes 
\[
\begin{split}
	{\cal P}(b) & =  J(\beta) \,{\mathbb P}(\sigma=\tau) +D(\beta) \,{\mathbb P}(\sigma<\tau)\\
	& = D(\beta)+(J(\beta)-D(\beta)){\mathbb P}(\sigma=\tau)\\
	& = D(\beta)     +(J(\beta)-D(\beta))   \left(\beta \,e^\beta \int_\beta^\infty \frac{e^{-z}}{z}{\rm d}z   \right).
\end{split}
\]
Finally, for the optimal $\beta^*$ we get with the account of $D(\beta^*)=\exp(-\beta^*)$

$$\sup_b{\cal P}(b) =e^{-\beta^*} +(e^{\beta^*}-1- \beta^*) \int_{\beta^*}^\infty \frac{e^{-s}}{s}{\rm d}s  =0.580164\cdots $$
which is the formula due to Samuels \cite{Handbook}.

Historically, the first study of the full-information best-choice problem with arrivals by Poisson process was Sakaguchi  \cite{SakaguchiFI}.
In that paper the marks are 
 uniform-$[0,1]$ and the process runs with finite horizon $T$. To obtain a sensible $T\to\infty$ limit one needs to 
resort to the equivalent model of Poisson process in $[0,1]\times[0,T]$. 
The finite-$T$ problem can then be interpreted as $N$ conditioned on the initial record at point $(0, T)$, then 
for $T\geq \beta^*$
the optimal success probability is  given by the above formula but with the upper limit $T$  in the exponential integral.

\section {A  uniform  triangular model}

In the models of this section the background process lives in the domain $x\geq t$. These have some appeal 
for applications in scheduling, where interval $[t,x]$ represents the time span needed to process a job by a server, and exactly one job is to be chosen by a stopping strategy.
The optimisation task is to maximise the probability of choosing the job 
with the earliest completion time.


\subsection{The discrete time problem}

Let $X_1,\cdots,X_n$ be independent, with $X_j$ having discrete  uniform distribution  on $\{j,\cdots,n\}$. 
Obviously, we  may focus on the states of  the running minimum within the lattice domain $\{(j,x): j\leq x\leq n\}$.

By Proposition \ref{as2} the optimal stopping time is given by a set  of  nondecreasing thresholds $b_j$.
Stopping at  record $(j,x)^\circ$ is successful with probability
\begin{equation}\label{s-jx}
s(j,x)= \prod_{i=0}^{x-j-1} \frac{n-x+1}{n-j-i}.
\end{equation}
Given the running minimum $M_j=x$ with $x\leq b_j$,  the continuation value 
is a specialisation of (\ref{v-in-B}), assuming the form

\begin{equation}\label{v-jx}
v(j,x) = \sum_{i=1}^{x-j} \left( \left(\prod_{k=0}^{i-2} \frac{n-x}{n-j-k} \right) \cdot \frac{1}{n-j-i+1} \cdot \sum_{y=j+i}^{x} s(j+i,y)\right).
\end{equation}
The success probability splits in two components, $v_n=J_n+D_n$, where $J_n$ results from  the running minimum breaking into $B$ by jump, 
while $D_n$ relates to drifting into $B$.
Explicitly,
$$J_n = \sum_{j=1}^{n} \left( \left( \prod_{k=0}^{j-2} \frac{n-b_j}{n-k} \right) \cdot \frac{1}{n-j+1} \cdot \sum_{x=j}^{b_j} s(j,x) \right)$$
and
$$D_n = \sum_{j=2}^{n}\left( \left( \prod_{i=0}^{j-2} \frac{n-b_{j-1}}{n-i} - \prod_{i=0}^{j-2} \frac{n-b_j}{n-i} \right) \cdot \sum_{y=b_{j-1}+1}^{b_j} \frac{v(j,y)}{b_j-b_{j-1}} \right),$$ 
  $b_j$ is the biggest $x$ with $s(j,x)\geq v(j,x)$ and the latter are given by  (\ref{s-jx}) and (\ref{v-jx}).

The computed values plotted in Figure \ref{fig:probs_triang_n} suggest that $v_n$ monotonically decreases  to a limit $0.703128\cdots$ (check the next subsection for its exact derivation).

\begin{figure}[!ht]
	\centering{
		\includegraphics[width = 0.6 \textwidth]{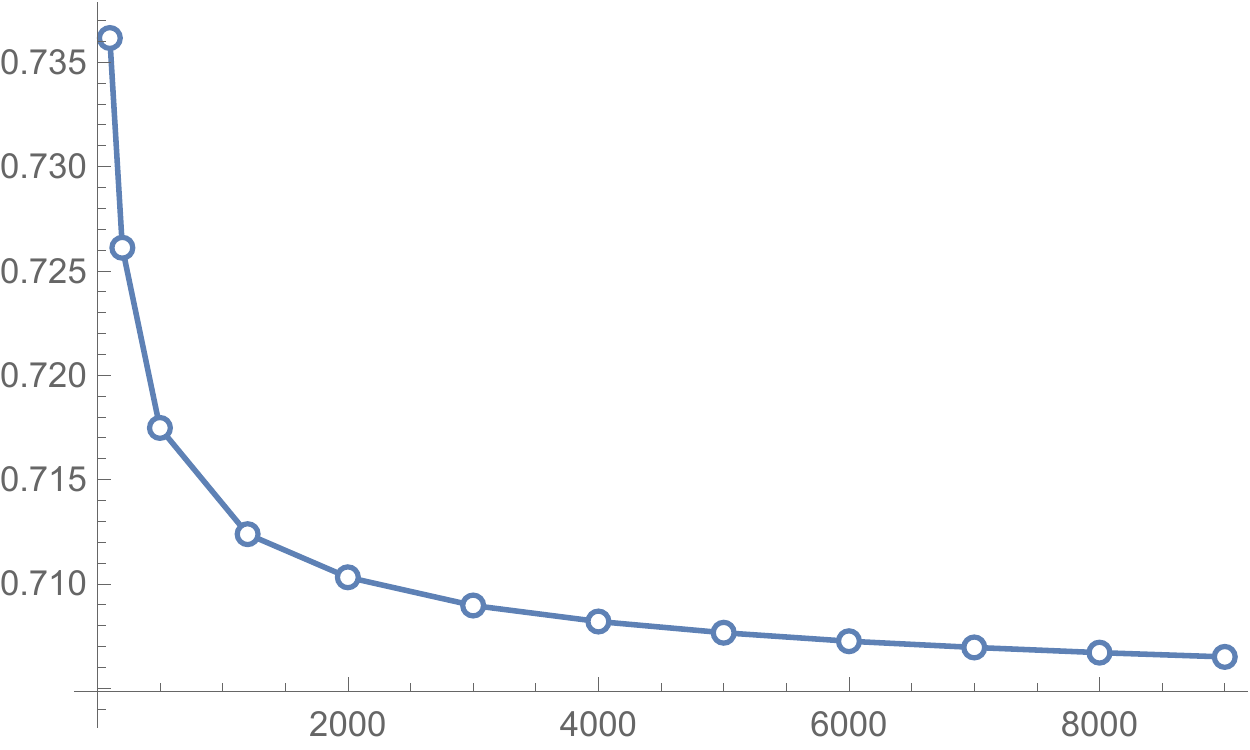}
		\caption{The optimal best-choice probability $v_n$  in the discrete triangular model for $n~ \in~\{100,\cdots,9000\}$.} \label{fig:probs_triang_n}
	}
\end{figure}

\subsection{Poissonisation}  \label{sec:triang}

The right scaling  is guessed from the Rayleigh distribution limit 
\begin{equation}\label{avoidance}
	\prob(M_n> x\sqrt{n})=\prod_{j \leq x \sqrt n}\left(1-\frac{j}{n-\lfloor x \sqrt{n} \rfloor +j}\right)\to e^{-x^2/2}, ~x>0.
\end{equation}
We truncated the product since
  $X_j>x \sqrt{n}$ for $j>x \sqrt{n}$. Thus we define  $N_n$  to be the point process with atoms $(j/\sqrt{n}, X_j/\sqrt{n}), ~j\leq n.$

Now, we assert that the process $N_n$ converges in distribution  to a  Poisson process $N$ with unit rate
in the sector $\{(t,x): 0\leq t\leq x<\infty\}$.  A pathway to proving this is the following.
For 
 $x>0$, 
convergence of the reduced point process of scaled times $\{j/\sqrt{n}: X_j\leq x\sqrt{n}\}$ is established in line with Chapter 9 of  \cite{Falk}: this includes convergence
of the mean measure and the avoidance probabilities akin to  (\ref{avoidance}) with $j$ in the bounds
 $t_1<j/\sqrt{n}<t_2$. 
Then the convergence of the planar process $N_n$ restricted to $[0,x]\times [0,x]$ follows by application of the theorem about marked Poisson processes. Sending $x\to\infty$ completes the argument.

\begin{figure}[!ht]
	\centering{
		\includegraphics[width = 0.4 \textwidth]{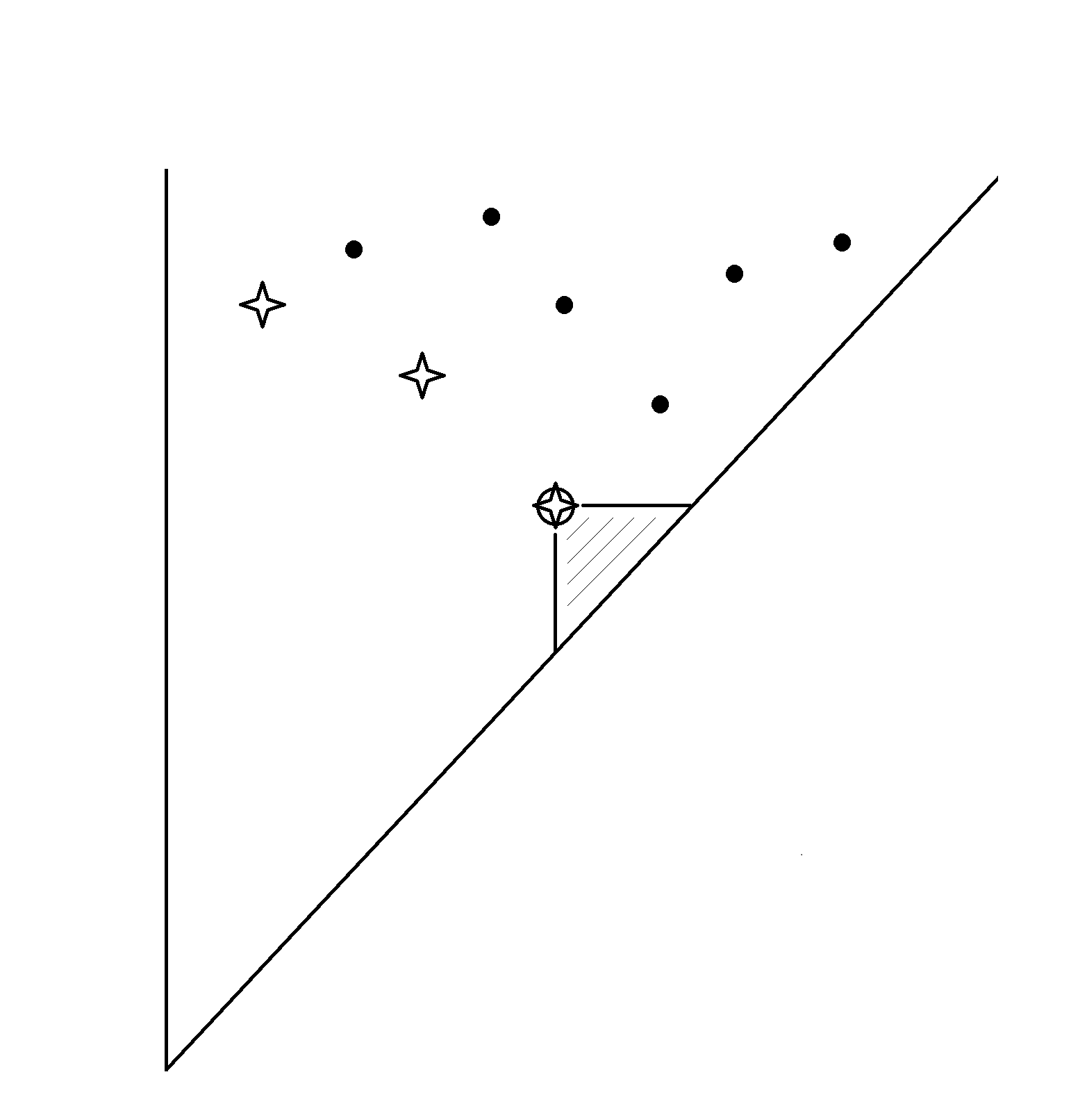}
		\caption[]{Records in a planar Poisson process above the diagonal $t=x$ of the positive quadrant; 		\scalerel*{\includegraphics{star_1.png}}{B} - record (no Poisson atoms south-west of it), 
		\scalerel*{\includegraphics{star_2.png}}{B} - the last record (its box contains no Poisson atoms).} \label{fig:record_triang}
	}

\end{figure}

The best-choice problem for $N$ is very similar to the one in the previous section. Under a box  with apex $(t,x)$ we shall understand now 
the isosceles triangle with one side lying on the diagonal and two other sides being parallel to the coordinate axis. 
The box area is equal to  $z:=(x-t)^2/{2}$.
Equal-sized  boxes can be mapped to one another by sliding along the diagonal (see Figure~\ref{fig:boundary_triang}).

\begin{figure}[!ht]
	\centering{
		\includegraphics[width = 0.47 \textwidth]{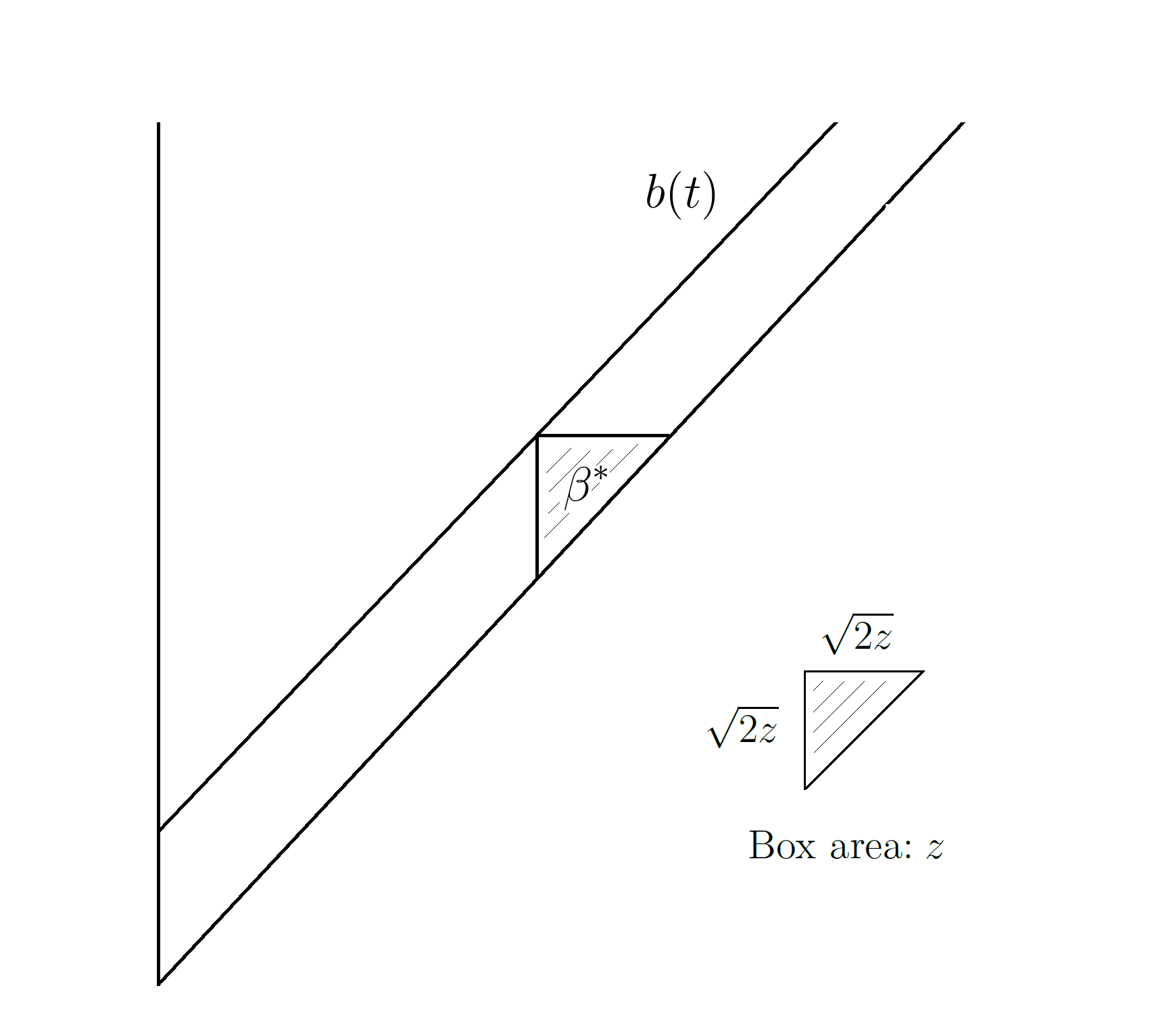}
		\caption{The linear boundary $b(t) = t + \sqrt{2 \beta^*}$.} \label{fig:boundary_triang}
	}
\end{figure}

In these terms, the basic functions are defined as follows: 

\begin{itemize}

\item[(i)] if stopping occurs on a record at $(t,x)$  it is successful with probability $e^{-z}$,

\item[(ii)] if stopping occurs on a record  at  random location $(t, t+U(x-t))$, for $U$ distributed uniformly on $[0,1]$,
 it is successful with probability 
\begin{eqnarray*}
J(z):=\int_0^1 e^{-z u^2}\,{\rm d}u= \frac{\sqrt{\pi}\erf(\sqrt{z})}{2\sqrt{z}}
\end{eqnarray*}
(recall that $\erf(x) = \frac{2}{\sqrt{\pi}} \int_0^x e^{-t^2} \dt$),

\item[(iii)] if stopping occurs on the earliest arrival  inside the box (if any)
 it is successful with probability 
\begin{eqnarray*}
D(z)&:=& \int_0^{\sqrt{2z}}   \exp\left( -\sqrt{2z}\,s+s^2/2  \right)   (\sqrt{2z}-s)     J\left( \left(\sqrt{z}-s/\sqrt{2}\right)^2  \right)\,{\rm d}s\\
&=&
e^{-z} \int_0^{\sqrt{2z}}\int_0^u e^{(u^2-v^2)/2}{\rm d}v{\rm d}u.
\end{eqnarray*}
\end{itemize}

The boundaries that come in question in this problem are non-decreasing functions $b:[0,\infty)\to[0,\infty)$ that satisfy $b(t)\geq t$. The  analogue of (\ref{CV}) becomes 

\begin{eqnarray*}
{\cal P}(b)&:=& \int_0^\infty  e^{-t b(t)+t^2/2}(b(t)-t)  J[(b(t)-t)^2/{2}]{\rm d} t  \\
&+&  \int_0^\infty  e^{-t b(t)+t^2/2}t   D[(b(t)-t)^2/{2}]{\rm d}b(t). 
\end{eqnarray*}
In the view of  self-similarity the maximiser should be a linear function
$$ b(t)=t+\sqrt{2\beta},$$
and then the success probability simplifies as 
\[
\begin{split}
	{\cal P}(b) & = D(\beta)+(J(\beta)-D(\beta)){\mathbb P}(\sigma=\tau)\\
	& = D(\beta)+(J(\beta)-D(\beta))    \sqrt{2\beta}  \int_0^\infty \exp\left({-\frac{t^2}{2}-\sqrt{2\beta} \,\,t}\right)\,{\rm d}t\\
	& = D(\beta)+(J(\beta)-D(\beta)) \sqrt{\pi\beta}\, e^\beta\, \erfc(\sqrt{\beta})
\end{split}
\]
(recall that $\erfc(x) = \frac{2}{\sqrt{\pi}} \int_x^\infty e^{-t^2} \dt$).
Equation $e^{-z}=D(z)$ becomes
$$\int_0^{\sqrt{2z}}\int_0^u e^{(u^2-v^2)/2}{\rm d}v{\rm d}u=1,$$
which by monotonicity has a unique solution
$$\beta^*=  0.760660\cdots$$
The stopping strategy with boundary $b(t)=t+\sqrt{2\beta^*}$ is overall optimal, and yields 
 the success probability
\begin{eqnarray} \label{OPT}
\sup_b {\cal P}(b)&=& e^{-\beta^*}+\left( e^{\beta^*} \frac{\sqrt{\pi}}{2\sqrt{\beta^*}} \erf\left(\sqrt{\beta^*}\right)-1   \right)\sqrt{\pi\beta^*} \erfc\left(\sqrt{\beta^*}\right)
\\
\nonumber
&=&0.703128\cdots,
\end{eqnarray}
confirming the limit we obtained numerically.

\subsection{The box-area jump chain and extensions}

This limit (\ref{OPT})  has appeared previously  in a context of  generalised records 
from partially ordered data \cite{Stoch}.  The source of coincidence lies in the  structure of the one-dimensional  box area process associated with the running minimum.
This is an interesting connection deserving some comments.

Consider a piecewise deterministic, decreasing Markov process  $P$  on ${\mathbb R}_+$, which drifts to zero at unit speed and jumps at unit rate.
When the jump occurs from location $z$, the new state is $zY$, where $Y$ is random variable with given distribution on $(0,1)$. The state $0$ is terminal. 
Thus if $P$ starts from $z>0$, in one drift-and-jump cycle the process moves to $(z-E)_+Y$, where $E$ is independent exponential random variable.
The associated optimisation problem amounts to stopping at the last state before absorption.

A process of this kind describes  a time-changed box area associated with the running minimum.
For the Poisson process of Section \ref{sec:GMfull},  the variable $Y$ is uniform-$[0,1]$, and in the triangular model it is beta$(1/2,1)$.
Two different modes of the first passage by the running minimum  occur when $P$ enters  $[0,\beta^*]$ by drift or by jump, where $\beta^*$ is the optimal parameter of the 
boundary.

More generally, for  
$Y$ following beta $(\theta,1)$ distribution,
Equation (9)   from \cite{Stoch} gives the success probability as
\begin{equation}\label{genTh}
{\cal P}(\beta^*) = \Gamma(-\theta+1,\beta^*,\infty) \left(-{\beta^*}^\theta\ + e^{\beta^*} \theta \Gamma(\theta,0,\beta^*)\right) + e^{-\beta^*},
\end{equation}
where
$$
\Gamma(a,b,c) = \int_{b}^{c} e^{-t}t^{a-1} \dt.
$$
One can verify  analytically that for $\theta=1/2$ the formula agrees with our (\ref{OPT}). Indeed,  (\ref{genTh})  specialises as 
\begin{equation}\label{th_half}
{\cal P}(\beta^*) = \Gamma(1/2,\beta^*,\infty) \left(-{\sqrt{\beta^*}} + \frac{e^{\beta^*}}{2} \Gamma(1/2,0,\beta^*)\right) + e^{-\beta^*},
\end{equation}
so observing
\begin{equation}\label{gamma_1}
\Gamma\left(1/2, \beta^{*}, \infty\right) = \int_{\beta^{*}}^{\infty}e^{-t}t^{-1/2} \dt = 2 \int_{\sqrt{\beta^{*}}}^{\infty} e^{-x^2} \dx = \sqrt\pi \erfc\left(\sqrt{\beta^{*}}\right)
\end{equation}
and similarly
\begin{equation}\label{gamma_2}
\Gamma\left(1/2, 0,  \beta^{*}\right) = \sqrt{\pi} \erf\left(\sqrt{\beta^{*}}\right)
\end{equation}
we obtain   (\ref{OPT}) by
substituting (\ref{gamma_1}) and (\ref{gamma_2}) into (\ref{th_half})

\vskip0.3cm

We also considered other processes of independent observations with linear trend that give the same limit best-choice probability (\ref{OPT}):
\begin{itemize}
\item[(i)] $X_1, \cdots, X_n$  independent, with $X_j$ distributed uniformly on $\{j,\cdots, j+n-1\}$.\\ 
Here again the limit distribution is Rayleigh, $\prob[M_n > x \sqrt{n}] \rightarrow e^{-x^2/2}$, 
and the point process with atoms $(j/\sqrt{n}, X_j/\sqrt{n})$ converges in distribution to a  Poisson process $N$ with unit rate in the  sector $\{(t,x): 0\leq t\leq x<\infty\}$.
\item[(ii)] $X_j=j+\rho n U_j$, where $\rho>0$ is a parameter and $U_1, U_2,\cdots$ are iid uniform-$[0,1]$.\\
This time $\prob[M_n > x \sqrt{\rho n}] \rightarrow e^{-x^2/2}$ and the point process with atoms  $(j/\sqrt{\rho n}, X_j/\sqrt{\rho n})$ converges weakly to  the same Poisson $N$.
\end{itemize}
On the other hand, 
\begin{itemize}
\item[(iii)] $X_j=j+ n U_j^{1/\theta}$, where $\theta>0$ is a parameter and $U_1, U_2,\cdots$ are iid uniform-$[0,1]$,\\
leads to (\ref{genTh}). Here   $\prob[M_n > x~n^\frac{\theta}{\theta +1}] \rightarrow e^{-\frac{x^{(\theta+1)}}{\theta+1}}$ which is a Weibull distribution with shape parameter $(\theta+1)$ and scale parameter $(\theta +1)^{\frac{1}{\theta +1}}$. The point process with atoms $(j/n^{\frac{\theta}{\theta +1}}, X_j/n^{\frac{\theta}{\theta +1}})$ converges weakly to the Poisson process which is not homogeneous, rather has intensity measure $\theta(x-t)^{\theta-1}{\rm d}t{\rm d}x, ~~~0\leq t\leq x$.
\end{itemize}

\section{A uniform rectangular model}

According to Proposition \ref{as_bounds_iid}, the limit best choice probability  for iid observations  is $\underline{v}$, provided the probability of a tie for the sample minimum approaches $0$ as $n\to\infty$.
For fixed,  not depending on $n$, 
discrete  distribution
this may or may not be the case. Moreover, when $(-X_j)$'s are geometric, the probability of a tie does not converge, but undergoes tiny fluctuations 
 \cite{Prodinger}; in this setting one can expect that the best choice probability has no limit  as well.
In this section we consider a discrete uniform distribution, and  achieve a positive limit probability of a tie for the sample minimum by letting the support of the distribution to depend on $n$.

\subsection{The discrete time problem} \label{sec:rect_disc}

Let $X_1, \cdots, X_n$ be independent, all distributed uniformly on $\{1,\cdots,n\}$. The generic state of the running minimum is a pair $(j,x)$, where $j,x \in \{1,\cdots,n\}$. In this setting the probability of a tie for a particular value does not go to $0$ with $n \to \infty$. In particular, the number of $1$'s in the sequence of $n$ observations is Binomial$(n,1/n)$, hence approaching  the Poisson$(1)$ distribution,
so the strategy which just waits for the first $1$ to appear succeeds with probability $1-(1-1/n)^n\to 1-1/e = 0.632120\cdots$, which already exceeds noticeably the universal sharp bound $\underline{v}~=~0.580164\cdots$ of Proposition \ref{as_bounds_iid}.

Again, by Proposition \ref{as2} the optimal stopping time is determined by a set of nondecreasing thresholds $b_1 \leq \cdots \leq b_n = \infty$. Stopping at record $(j,x)$ is successful with probability
$$
s(j,x) = \left(\frac{n-x+1}{n}\right)^{n-t}.
$$
Conditionally on the running minimum $M_j=x$ with $x<b_j$, the continuation value given by~(\ref{v-in-B}) reads as
$$
v(j,x) = \sum_{i=1}^{n-t}  \left(\frac{n-x}{n}\right)^{i-1} \frac{1}{n} \sum_{y=1}^x  s(t+i,y).
$$



\begin{figure}[!ht]
	\centering{
		\includegraphics[width = 0.6 \textwidth]{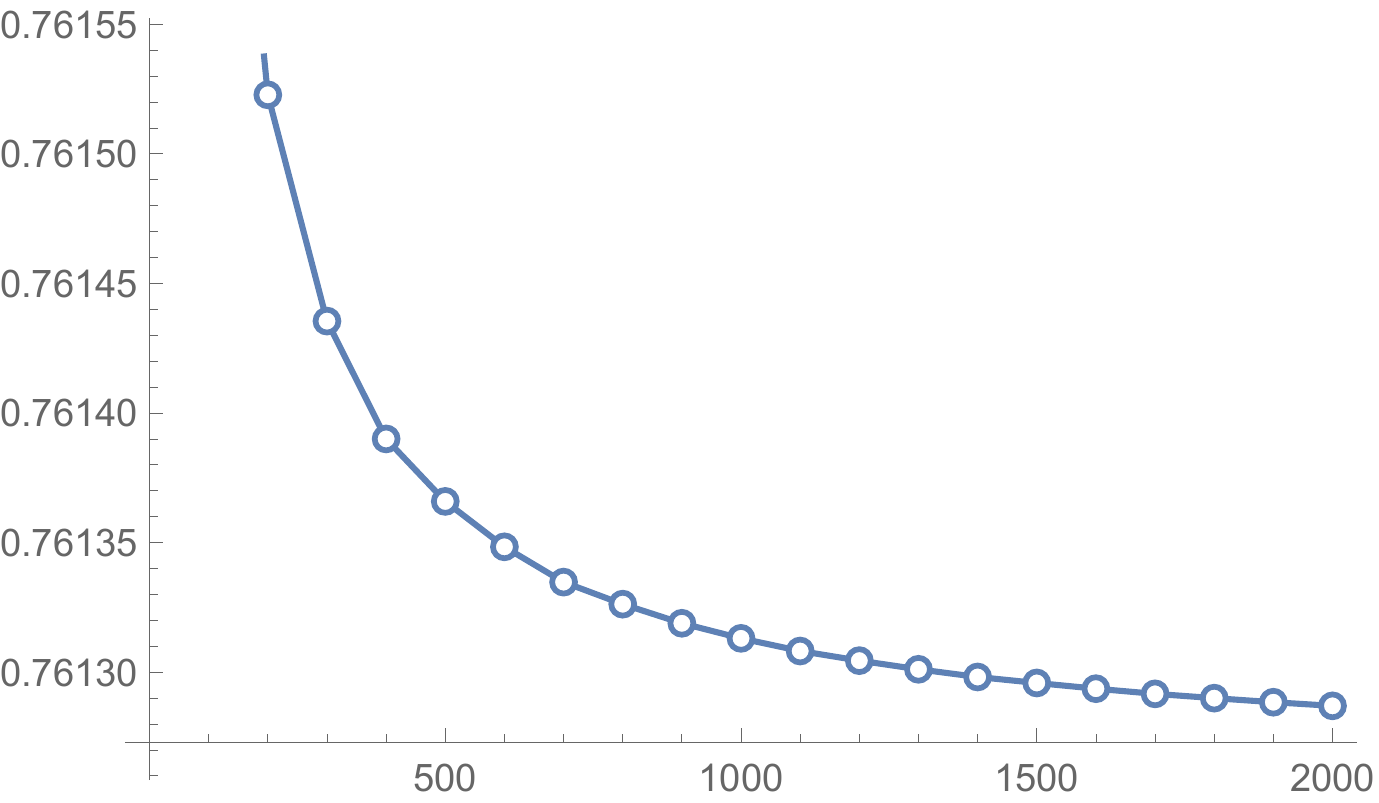}
		\caption{The optimal best-choice probabilities in the discrete rectangular model for $n \in\{100,200,\cdots,2000\}$.} \label{fig:probs_rect_n}
	}
\end{figure}

The success probability may be again decomposed into terms $v_n = J_n + D_n$, referring to the running minimum entering  $B$ by jump or by drift, respectively. We get
$$
J_n=\sum_{j=1}^n \left(\frac{n-b_j}{n}\right)^{j-1} \frac{1}{n}\sum_{x=1}^{b_j}s(j,x),
$$
and
$$
D_n=\sum_{j=2}^n\left[\left(\frac{n-b_{j-1}}{n}\right)^{j-1} -  \left(\frac{n-b_j}{n}\right)^{j-1} \right] \sum_{y=b_{j-1}+1}^{b_j}   \frac{v(j,y)}{b_j-b_{j-1}},
$$ 
where $b_j$ is defined as the biggest $x$ with $s(j,x) \geq v(j,x)$, and these are given by  (\ref{s-jx}) and (\ref{v-jx}).
The computed values, as  presented in Figure \ref{fig:probs_triang_n}, suggest that $v_n$  decreases monotonically to a limit $0.761260\cdots$.
Using the  Poisson approximation we shall obtain an explicit expression in terms of the roots of certain equations.


\subsection{Poissonisation} \label{sec:poiss_rect}

The point process $(j/n, X_n)$ converges to a
Poisson process on $[0,1]\times {\mathbb Z}_{>0}$ with the intensity measure being the product of Lebesgue measure and the counting measure on integers.
Hence to find the limit success probability we may work directly with the setting of this Poisson process. We prefer, however, to stay within the  continuous framework of  previous sections, and to work with the planar Poisson point process $N$ in $[0,1]\times[0,\infty)$ with the Lebesgue measure as intensity.

To that end, we just modify the ranking order.
Let $X_1, \cdots, X_n$ be iid uniform-$[0,n]$. Two values with $\lceil X_j\rceil = \lceil X_i \rceil$ will be treated as order-indistinguishable. In particular, we call $X_j$ a (weak) record if $\lceil X_j\rceil \leq \lceil X_i\rceil$ for all $i<j$. For $n$ large, the distribution of $\lceil M_n \rceil$ is close to Geometric$(1-1/e)$.

Now, the planar point process with atoms $(j/n, X_j)$, $j \leq n$, converges in distribution to $N$. The running minimum $(Z_t, t \in [0,1)])$ is the lowest mark of arrival on $[0,t]$. Marks $x,y$ with the same integer ceiling $\lceil x\rceil= \lceil y \rceil$ will be considered as order-indistinguishable.
Accordingly, arrival $(t,x)$ is said to be a (weak) record if $[0,t]\times [0,\lceil x\rceil)$ contains no Poisson atoms.
The role of a box is now played by the rectangle  $[t,1]\times [0,\lceil x\rceil)$.

The basic functions are defined as follows:
\begin{itemize}
	
	\item[(i)] if stopping occurs on a record at $(t,x)$  with $\lceil x\rceil=k$ it is successful with probability $$\exp(-(1-t)(k-1)),$$
	
	\item[(ii)] if stopping occurs on the earliest arrival (if any) inside the box $[t,1]\times [0,\lceil x\rceil)$ with $\lceil x\rceil=k$ it is successful with probability
	\begin{eqnarray*}
		\int_{t}^{1} e^{-k (s-t)} \sum_{j=1}^k e^{-(j-1)(1-s)} {\rm d}s=
		e^{-k(1-t)} \sum_{j=1}^k \frac{e^{j(1-t)}-1}{j}.
	\end{eqnarray*}

\end{itemize}

For $k=1$ stopping is optimal for all $t$; we set $t_1=0, z_1=e$ and for $k\geq 2$ the equality is achieved for $t_k$  defined to be the root of equation 
$$ e^{-(k-1)(1-t)}=e^{-k(1-t)}\sum_{j=1}^k  \frac{e^{j(1-t)}-1}{j}.$$
Letting $z_k:=e^{1-t_k}$, $z_k$ is a solution to

\begin{equation}\label{EqZ}
	\sum_{j=2}^k \frac{z^j}{j}=h_k, ~~~h_k:=\sum_{j=1}^k \frac{1}{j}.
\end{equation}
By monotonicity there exists a unique positive solution, and the roots are decreasing, so that  $z_1=e$ and  $z_k\downarrow 1$.

It follows that the optimal stopping time is
$$\tau=\inf\{t:~(t,Z_t)~{\rm is ~a~record}, {\rm ~such~that~} t \geq  1-\log(   z_{k})~{\rm for~}k=\lceil Z_t\rceil\}.$$ 
That is, the stopping boundary is
$$b(t) =\sum_{k=1}^{\infty} k \,1(t_k<t\leq t_{k+1})= \sum_{k=1}^{\infty} k \,\,1(z_{k+1}\leq e^{1-t} < z_k).$$
The cutoffs are  readily computable from (\ref{EqZ}), for instance
\begin{eqnarray*}
	z_1=e=2.71828\cdots,\\
	z_2=\sqrt{3}=1.732050\cdots,\\
	z_3= 1.381554\cdots,\\
	z_4=1.258476\cdots,\\
	z_5=1.195517\cdots,\\
	z_{10}= 1.088218\cdots,\\
	z_{15}=1.056969\cdots,\\
	z_{20}=1.042069\cdots.
\end{eqnarray*}

The associated hitting time for the running minimum is
$$\sigma=\inf\{t: \lceil Z_t\rceil\leq b(t)\}.$$

The success probability again decomposes in terms corresponding to the events $\sigma<\tau$ and $\sigma=\tau$. The first term related to jump through the boundary becomes

\[
\begin{split}
	J & :=\sum_{k=1}^\infty     \int_{t_k}^{t_{k+1}} e^{-k t} \sum_{j=1}^k e^{-(j-1)(1-t)}{\rm d} t \\
	& = \sum_{k=1}^\infty\sum_{j=1}^k \frac{e^{-k}}{j}\, (e^{j(1-t_k)}-e^{j(1-t_{k+1})}) \\
	& = \sum_{k=1}^\infty \sum_{j=1}^k  \frac{e^{-k}}{j} (z_k^j-z_{k+1}^j)\\
	& = e^{-1}(z_1-z_2)+\sum_{k=2}^\infty e^{-k} \left(z_k-z_{k+1} +\frac{z_{k+1}^{k+1}-1}{k+1}\right),
\end{split}
\]

where for the last equality we used (\ref{EqZ})  in the form
\begin{eqnarray*}
	\sum_{j=1}^k \frac{z_{k}^j}{j}&=&h_k+z_k,\\
	\sum_{j=1}^k \frac{z_{k+1}^j}{j}&=&h_{k+1} +z_{k+1}  -\frac{z_{k+1}^{k+1}}{k+1}.  
\end{eqnarray*}

Note that if the ceiling of the running minimum $\lceil Z\rceil$ drifts into the boundary point $(t_k, k)$, the optimal success probability from this time on is the same as from stopping as if 
a record occurred at $(t_k,k)$. Hence the contribution of the event $\sigma<\tau$ becomes

\begin{eqnarray*}
	D:=    \sum_{k=2}^\infty (e^{-(k-1)t_k} - e^{-k t_k})e^{-(k-1)(1-t_k)} =\sum_{k=2}^\infty  e^{-k} (e-z_k).
\end{eqnarray*}

Putting the parts together, the optimal success probability after some cancellation and series work becomes
\[
\begin{split}
	{\mathbb P}(Z_\tau =\min_{t\in[0,1]} Z_t) & =J+D \\
	& = 2 -e +\frac{1}{e-1} +\frac{1-2\sqrt{3}}{2e} +e\log(e-1)
	-\sum_{k=2}^\infty e^{-k}\left(z_{k+1}- \frac{z_{k+1}^{k+1}}{k+1}\right)\\
	& = 0.761260\cdots
\end{split}
\]

\paragraph{The general boundary} For the general boundary defined by nondecreasing cutoffs $t_k$, 
the jump term 
$$
J:=\sum_{k=1}^\infty \sum_{j=1}^k  \frac{e^{-k}}{j} (z_k^j-z_{k+1}^j)
$$
should be computed with $z_k=e^{1-t_k}$, and
the drift term  written as

\[
\begin{split}
	D & :=    \sum_{k=2}^\infty (e^{-(k-1)t_k} - e^{-k t_k})         e^{-k(1-t_k)}\sum_{j=1}^k \frac{e^{j(1-t_k)}-1}{j} \\
	& = \sum_{k=2}^\infty e^{-k} (e^{t_k} - 1)   \sum_{j=1}^k \frac{e^{j(1-t_k)}-1}{j} \,.
\end{split}
\]
For instance, letting $t_k=1$ for $k \geq 3$, the maximum success probability is  $0.730694\cdots$ achieved at $t_2=0.450694\cdots$.


\subsection{Varying the intensity of  the Poisson process}
The extension presented in this section constitutes a smooth transition between the above poissonised rectangular model and the full-information game from Section \ref{sec:GMfull}. 
As above, consider a homogeneous  Poisson process on $[0,1]\times[0,\infty)$, and treat   values $x,y$ with $\lceil x\rceil= \lceil y \rceil$ as order-indistinguishable,
but now suppose the intensity of the process is  some $\lambda>0$. 
Note that as $\lambda\to 0$ the ties vanish 
hence the best-choice probability becomes close to $\underline{v} = 0.580164\cdots$ from the full-information game.

This process relates to a limit form of the  discrete time best-choice problem, with observations 
$X_1, \cdots, X_n$ drawn from the  uniform distribution on $\{1,\cdots,K_n\}$ where $K_n \sim n / \lambda$.
See \cite{Falk} (Example 8.5.2) and \cite{Kolchin} for the related extreme-value theory.
Here, for $n$ large, the distribution of $M_n$ is close to Geometric($1-(1/e)^{\lambda}$). 
The scaling dictated by convergence to the Poisson limit is $(j/n, \lambda X_j)$.

Following the familiar path, we compare
stopping on a record $(t,x)$, for given $\lceil x \rceil = k$, with  stopping on the next available record.
For $k=1$ stopping is the optimal action for all $t$. We set $t_1=0$ and $z_{1}^{(\lambda)} := e^{\lambda}$. 
For $k \geq 2$, whenever a positive solution to
\begin{equation}
	\sum_{j=2}^k \frac{z^j}{j}=h_k, ~~~h_k:=\sum_{j=1}^k \frac{1}{j}
\end{equation}
is smaller than $e^{\lambda}$, we define $z_{k}^{(\lambda)}$ to be this solution and set $z_{k}^{(\lambda)} := e^{\lambda(1-t_k)}$. Otherwise, we set $z_{k}^{(\lambda)} = e^{\lambda}$ (which corresponds to setting the threshold $t_k$ to $0$). The optimal stopping time is thus given by 
$$\tau=\inf\left\{t:~(t,Z_t)~{\rm is ~a~record} {\rm ~with~} t \geq  1-\frac{\log(   z_{k}^{(\lambda)})}{\lambda}{\rm ~for~} k=\lceil Z_t\rceil\right\}.$$
Equivalently,  the stopping boundary is
$$b(t) =\sum_{k=1}^{\infty} k \,1(t_k<t\leq t_{k+1})= \sum_{k=1}^{\infty} k \,\,1(z_{k+1}^{(\lambda)}\leq e^{\lambda(1-t)} < z_{k}^{(\lambda)}).$$

The optimal success probability decomposes into the jump and drift terms: 
$${\mathbb P}(Z_\tau=\min_{t\in[0,1]} Z_t) = J_{\lambda} + D_{\lambda},$$ 
where

\begin{eqnarray*}
J_{\lambda} &=& \sum_{k=1}^\infty \sum_{j=1}^k  \frac{e^{-\lambda k}}{j} \left((z_k^{(\lambda)})^j-(z_{k+1}^{(\lambda)})^j\right), \\
	D_{\lambda} &=& \sum_{k=0}^{\infty} e^{-\lambda k} (e^{\lambda}-z_k^{(\lambda)}).
\end{eqnarray*}
The numerical values of the best choice probability 
are plotted in Figure \ref{fig:lambdas}.


\begin{figure}[!ht]
	\begin{subfigure}{.5\textwidth}
		\centering
		\includegraphics[width=.9\linewidth]{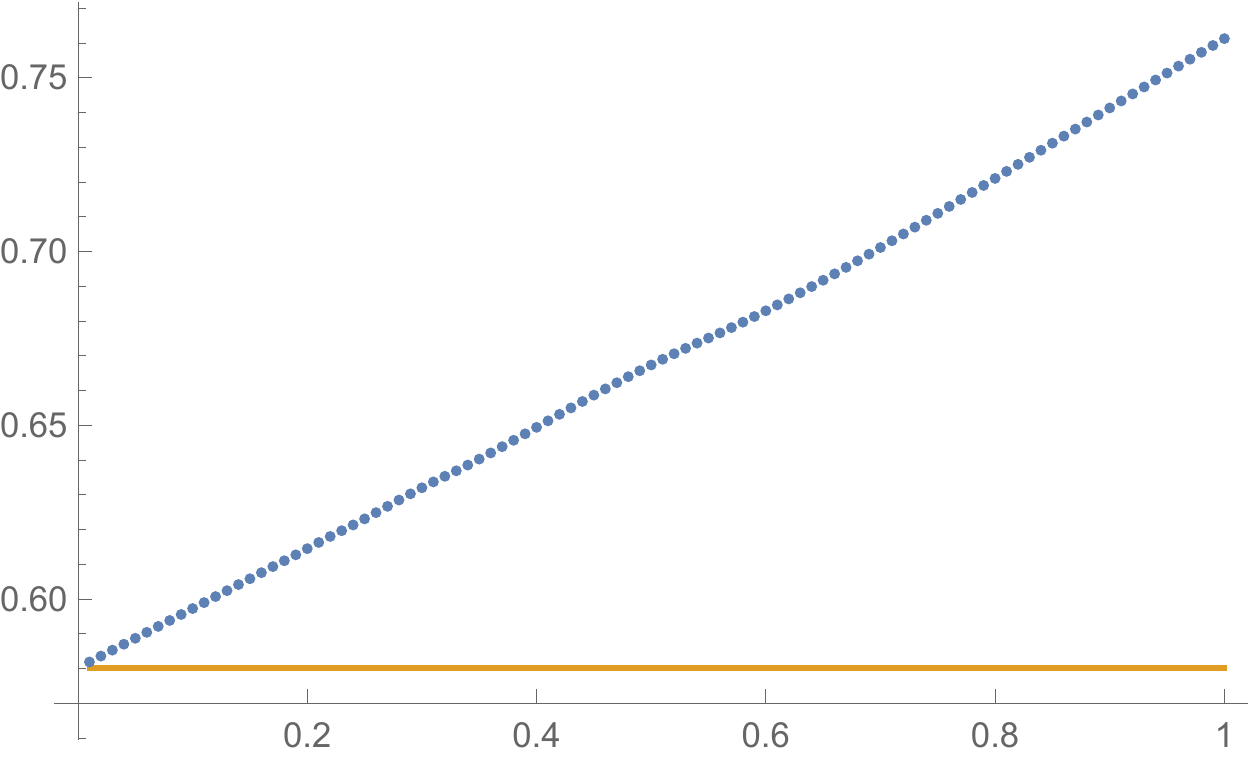}
		\caption{$\lambda \in \{0.01, 0.02, \cdots, 1\}$} 
	\end{subfigure}
	\hfill
	\begin{subfigure}{.5\textwidth}
		\centering
		\includegraphics[width=.9\linewidth]{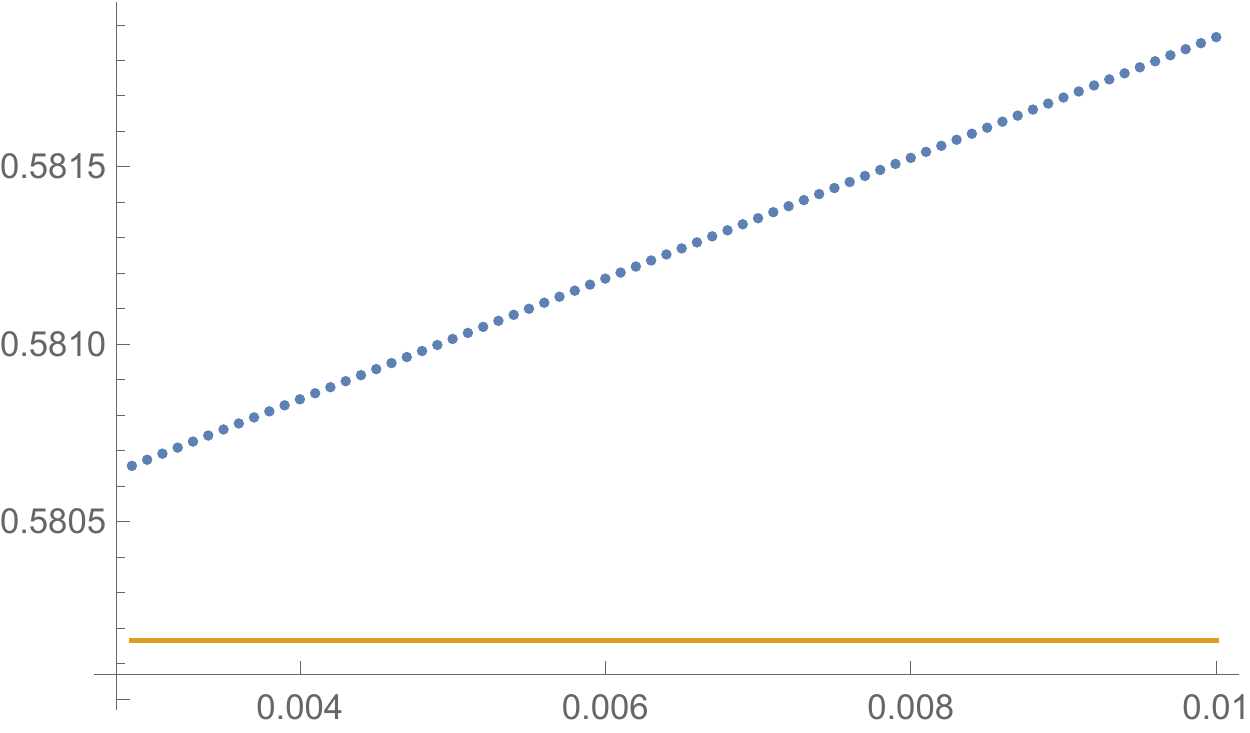}
		\caption{$\lambda \in \{0.003, 0.0031, \cdots, 0.01\}$}
	\end{subfigure}
	\caption{The success probability values for different ranges of $\lambda$,   in comparison with the benchmark success probability  $0.580164\cdots$ from the full-information game.}
	\label{fig:lambdas}
\end{figure}



\begin{thebibliography}{99}


\bibitem{Tie}
Baryshnikov, Y., Eisenberg, B. and Stengle, G. (1995)
A necessary and sufficient condition for the existence of the limiting probability of a tie for first place, 
{\it Stat.  Probab. Lett.} {\bf 23}(3), 203--209.



\bibitem{Campbell}
Campbell, G.  (1982)
The maximum of a sequence with prior information,
{\it Comm. Stat. (Part C: Sequential Analysis)} {\bf 1}(3), 177--191.

\bibitem{CRS} 
Chow, Y.S., Robbins, H. and Siegmund, D.
The theory of optimal stopping (2d edition),
Dover, 1991.

\bibitem{Falk}  
Falk, M., H{\"u}sler, J. and Reiss, R.-D. 
Laws of small numbers: extremes and rare events,
Birkh{\"a}user, 2011.

\bibitem{FR}
Faller, A. and R{\"u}schendorf, L. (2012) 
Approximative solutions of best choice problems,
{\it Elec. J. Probab.} {\bf 17}, 1--22.

\bibitem{Ferguson}
Ferguson, T.S. (1989)
Who solved the secretary problem?
{\it Statist. Sci.} {\bf 4}(3), 282--289.

\bibitem{Gilbert_Mosteller}
Gilbert, G.P. and Mosteller, F. (1966)
Recognizing the maximum of a sequence,
{\it J.~Am. Stat. Assoc.} {\bf 61}(313), 35--73.

\bibitem{GnedinFI}
Gnedin, A. (1996)
On the full-information best-choice problem,
{\it J. Appl. Probab.} {\bf 33}(3), 678--687.

\bibitem{RCS}
Gnedin, A. (1997)
The representation of composition structures,
{\it Ann. Probab.} {\bf 25}(3), 437--1450. 

\bibitem{GnedinPlanar}
Gnedin, A.  (2004)
Best choice from the planar Poisson process,
{\it Stoch. Proc. Appl.} {\bf  111}, 317--354.

\bibitem{GnKr}
Gnedin, A. and Krengel, U. (1995)
A stochastic game of optimal stopping and order selection,
{\it Ann. Appl. Probab.} {\bf 5}, 310--321.



\bibitem{Stoch}
Gnedin, A. (2007)
Recognising the last record of a sequence,
{\it Stochastics} {\bf 79}, 199--209.

\bibitem{HK}
Hill, T.P. and Kennedy, D.P. (1992)
Sharp inequalities for optimal stopping with rewards based on ranks,
{\it Ann.  Appl. Probab.} {\bf 2}, 503--517.




\bibitem{Prodinger} Kirschenhofer, P. and Prodinger, H. (1996) The number of winners in a discrete geometrically distributed sample,
{\it  Ann. Appl. Probab.} {\bf 6}(2), 687--694.

\bibitem{Kolchin}
Kolchin, V.F. (1969) The limiting behavior of extreme terms of a variational
series in polynomial scheme, {\it  Theory Probab. Appl.} {\bf 14}, 458--469.



\bibitem{Nuti}
Nuti, P. (2020)
On the best-choice prophet secretary problem,
{\tt arXiv:2012.02888}



\bibitem{Resnick}
Resnick, S.I.
Extreme values, regular variation and point processes, Springer, 2008. 

\bibitem{Sak}
Sakaguchi, M. (1973)
A note on the dowry problem,
{\it Rep. Statist. Appl. Res. Un. Japan. Sci. Eng.} {\bf 20}, 11--17.

\bibitem{SakaguchiFI}
Sakaguchi, M. (1976)
Optimal stopping problems for randomly arriving offers,
{\it Math. Japon.}, {\bf 21}(2), 201--217.

\bibitem{Handbook}
Samuels, S.M. (1991)
Secretary problems, in: {\it Handbook of sequential analysis}, B.K. Ghosh and P.K. Sen (eds), 381--405.

\bibitem{Tamaki}
Tamaki, M. (2015)
On the optimal stopping problems with monotone thresholds,
{\it J.~Appl. Probab.} {\bf  52}(4), 926--940.


\end{thebibliography}

\end{document}